\documentclass[11p]{article}
\usepackage[a4paper, total={6in, 9in}]{geometry}

\usepackage[normalem]{ulem}
\usepackage{natbib}
\usepackage{apalike}
\usepackage{setspace}
\usepackage{amsmath}
\usepackage{amssymb}
\usepackage{stmaryrd}
\SetSymbolFont{stmry}{bold}{U}{stmry}{m}{n}
\usepackage{comment}
\usepackage{wrapfig}

\usepackage{bbold}
\usepackage{dsfont}
\usepackage{tikz}
\usepackage{hyperref}
\usepackage{amsthm}
\newtheorem{theorem}{Theorem}

\newtheorem{proposition}{Proposition}
\newtheorem{corollary}{Corollary}

\theoremstyle{definition}

\allowdisplaybreaks
\usepackage{algpseudocode}
\usepackage{algorithm}
\usepackage{url}
\usepackage{cleveref}

\usepackage{todonotes}
\usepackage{mathtools} 
\usepackage{longtable,booktabs}
\usepackage{csvsimple,pgfkeys}
\usepackage{pdflscape}
\usepackage{graphicx}
\usepackage{xcolor}

\usepackage{scalerel}
\usepackage{tikz}
\usetikzlibrary{svg.path}

\definecolor{orcidlogocol}{HTML}{A6CE39}
\tikzset{
	orcidlogo/.pic={
		\fill[orcidlogocol] svg{M256,128c0,70.7-57.3,128-128,128C57.3,256,0,198.7,0,128C0,57.3,57.3,0,128,0C198.7,0,256,57.3,256,128z};
		\fill[white] svg{M86.3,186.2H70.9V79.1h15.4v48.4V186.2z}
		svg{M108.9,79.1h41.6c39.6,0,57,28.3,57,53.6c0,27.5-21.5,53.6-56.8,53.6h-41.8V79.1z M124.3,172.4h24.5c34.9,0,42.9-26.5,42.9-39.7c0-21.5-13.7-39.7-43.7-39.7h-23.7V172.4z}
		svg{M88.7,56.8c0,5.5-4.5,10.1-10.1,10.1c-5.6,0-10.1-4.6-10.1-10.1c0-5.6,4.5-10.1,10.1-10.1C84.2,46.7,88.7,51.3,88.7,56.8z};
	}
}
\newcommand\orcidicon[1]{\href{https://orcid.org/#1}{\mbox{\scalerel*{
				\begin{tikzpicture}[yscale=-1,transform shape]
				\pic{orcidlogo};
				\end{tikzpicture}
			}{|}}}}

\newcommand{\tr}{\text{tr}}
\newcommand{\Diag}{\text{Diag}}
\newcommand{\diag}{\text{diag}}

\newcommand{\rank}{\text{rank}}

\newcommand{\aff}{\text{aff}}
\newcommand{\conv}{\operatorname{conv}}

\DeclareMathOperator*{\argmin}{arg\,min}

\DeclarePairedDelimiter{\set}{\{}{\}}
\DeclarePairedDelimiter{\roundbr}{(}{)}

\newcommand{\ones}{\ensuremath{\mathbf{1}}}
\newcommand{\zeros}{\ensuremath{\mathbf{0}}}
\newcommand{\mat}[1]{\ensuremath{\mathbf{#1}}}
\newcommand{\mattran}[1]{\ensuremath{{#1}^\top}}
\newcommand{\indvec}[1]{\mathbb{1}_{#1}}
\newcommand{\R}{\mathbb{R}}
\newcommand{\algconnPn}{\ensuremath{2\left(1-\cos\left( \frac \pi n \right)\right)}}

\newcommand{\Ylarge}{\ensuremath{\widetilde{Y}}}
\newcommand{\Qlarge}{\ensuremath{\widetilde{Q}}}

\newcommand{\revision}[1]{{#1}}

\begin{document}

\title{Spanning and Splitting: Integer Semidefinite Programming for the Quadratic Minimum Spanning Tree Problem}

\author{
Frank de Meijer\footnote{Delft Institute of Applied Mathematics, Delft University of Technology, Mekelweg 4, 2628 CD Delft, The Netherlands,
\href{mailto:f.j.j.demeijer@tudelft.nl}{f.j.j.demeijer@tudelft.nl}}~\footnote{Corresponding Author:\href{mailto:f.j.j.demeijer@tudelft.nl}{f.j.j.demeijer@tudelft.nl}}~\orcidicon{0000-0002-1910-0699} 
\and
Melanie Siebenhofer\footnote{Institut f\"ur Mathematik, Alpen-Adria-Universit\"at Klagenfurt, Universit\"atstra{\ss}e 65-67, 9020 Klagenfurt, \href{mailto:melanie.siebenhofer@aau.at}{melanie.siebenhofer@aau.at}, \href{mailto:angelika.wiegele@aau.at}{angelika.wiegele@aau.at}}~\footnote{This research was funded in part by the Austrian Science Fund (FWF) [10.55776/DOC78]. For open access purposes, the authors have applied a CC BY public copyright license to any author-accepted manuscript version arising from this submission.}~\orcidicon{0000-0002-9101-834X}
\and 
Renata Sotirov\footnote{Tilburg University, Department of Econometrics \& Operations Research, CentER, 5000 LE Tilburg, 
\href{mailto:R.Sotirov@tilburguniversity.edu}{r.sotirov@tilburguniversity.edu}}~\orcidicon{0000-0002-3298-7255} 
\and Angelika Wiegele$^{\ddagger\S}$\footnote{Universität zu Köln, Weyertal 86--90, 50931 Köln, Germany}~\orcidicon{0000-0003-1670-7951}
}

\date{ }

\maketitle

\begin{abstract}
  In the quadratic minimum spanning tree problem (QMSTP) one wants to find the minimizer of a quadratic function over all possible spanning trees of a graph.
  We {present a formulation of the QMSTP as a mixed-integer semidefinite
  program} exploiting the algebraic connectivity of a graph. Based on
  {this formulation}, we derive a doubly nonnegative
  relaxation for the QMSTP and investigate classes of valid inequalities to
  strengthen the relaxation using the Chvátal-Gomory procedure for
  mixed-integer conic programming.
  
  Solving the resulting relaxations is out of reach for off-the-shelf
  software. We therefore develop and implement a version of the
  Peaceman-Rachford splitting method that allows to compute the
  new bounds for graphs from the literature.
  The computational results demonstrate that our bounds significantly
  improve over existing bounds from the literature in both quality and computation time, in particular for
  graphs with more than 30~vertices.
  
  This work is further evidence that semidefinite programming is a
  valuable tool to obtain high-quality bounds for problems in
  combinatorial optimization, in particular for those that can be
  modelled as a quadratic problem.  
\end{abstract}


\section{Introduction}

The quadratic minimum spanning tree problem (QMSTP) is the problem of finding a spanning tree of a connected, undirected graph such that the sum of interaction costs over all pairs of edges in the tree is minimized.
The QMSTP was introduced by~\citet{AssadXu} in~1992. 
The adjacent-only  quadratic minimum spanning tree problem (AQMSTP),
that is, the QMSTP where the interaction costs of all non-adjacent edge pairs are assumed to be zero, is also introduced in \citep{AssadXu}.
Assad and Xu proved that both the QMSTP and AQMSTP are strongly ${\mathcal NP}$-hard problems. Interestingly, the QMSTP remains ${\mathcal NP}$-hard even when the cost matrix is of rank one \citep{Punnen2001}.

There are many existing variants of the QMSTP problem, such as the minimum spanning tree problem with conflict pairs,  the quadratic bottleneck spanning tree problem, and the bottleneck spanning tree problem with conflict pairs.  For a description of those problems, see e.g.,  \'Custi\'c et al.~\citep{CUSTIC201873}.
 The QMSTP has various applications in telecommunication,  transportation, energy and hydraulic
networks, see e.g., \citep{AssadXu,CHIANG2007734,Chou1973AUA}.

There is a lot of research on lower-bounding approaches and exact algorithms for the QMSTP. The majority of lower bounding approaches for the QMSTP may be classified into Gilmore-Lawler (GL) type bounds \citep{AssadXu,Cordone2012SolvingTQ,ONCAN20101762,Rostami2015LowerBF} and
reformulation linearization technique (RLT) based bounds \citep{PEREIRA2015149,Rostami2015LowerBF}.
The GL procedure is a well-known approach to construct lower bounds for quadratic binary optimization problems, see e.g.,~\citep{Gilmore,Lawler}.
The RLT is a method to derive a hierarchy of convex approximations of mixed-integer programming problems~\citep{sherali2013reformulation} where integer variables are binary.
Lower bounding approaches based on an extended formulation of the minimum spanning tree problem are derived in~\citep{Sotirov2024}.
For an overview of the above-mentioned lower bounding approaches and their comparison, see e.g.,~\citep{Sotirov2024}.
Semidefinite programming (SDP) lower bounds for the QMSTP are considered in~\citep{GUIMARAES202046}. SDP bounds incorporated in a branch-and-bound algorithm provide the best exact solution approach for the problem up to date~\citep{GUIMARAES202046}.  Different exact approaches for solving the QMSTP are considered in~\citep{AssadXu,Cordone2012SolvingTQ,PEREIRA2015149,Pereira2015BranchandcutAB}. For a comparison of various heuristic approaches for solving the QMSTP, see~\citet{PalubeckisEtAl}.

In this paper, we derive a compact mixed-integer semidefinite (MISDP) formulation for the QMSTP by exploiting the algebraic connectivity of a tree. 
Algebraic connectivity was also exploited in~\citep{Cvetkovic, de2022chv} to derive ISDP  formulations for the traveling salesman problem (TSP) and the quadratic TSP, respectively.
 This formulation contributes to the recent development in mixed-integer semidefinite programming, revealing that many
hard combinatorial optimization problems have compact MISDP formulations~\citep{de2023misdp}. 
We prove that the continuous relaxation of our MISDP formulation is not dominated by the continuous relaxation of the cut-set QMSTP formulation, and vice versa.
Further, we  derive several classes of valid inequalities for our MISDPs by exploiting the Chv\'atal-Gomory  (CG) procedure for mixed-integer conic  programming~\citep{CezikIyengar,de2022chv}. 
In particular, we show that the classical cut-set constraints and the first-level RLT constraints are CG cuts. Although these cuts have been considered before for the QMSTP, obtaining them as CG cuts provides a measure for how well the continuous relaxation of our MISDP formulation describes the convex hull of feasible solution vectors.  
The cut-set constraints are derived from the linear matrix inequality (LMI) that is related to the algebraic connectivity of a tree. The RLT-type constraints are derived using two LMIs from the MISDP formulation of the QMSTP.

We exploit our MISDP formulation for the QMSTP to derive a doubly nonnegative (DNN) relaxation of the problem. DNN relaxations  have matrix variables that are both positive semidefinite and nonnegative.
Our DNN relaxation for the MISDP does not include the connectivity constraint in the form of a linear matrix inequality, because that constraint has only a small impact on the bound. Moreover,
preliminary computational results show that the cut-set constraints also have a small impact on the quality of our doubly nonnegative  relaxation of the QMSTP. However, the RLT-type constraints improve the DNN bound. Therefore, we add  RLT-type constraints to the DNN relaxation of the QMSTP.
Although our relaxation is theoretically dominated by the relaxation of~\cite{GUIMARAES202046}, we carefully selected which constraints to include in the relaxation in order to achieve an optimal balance between solution quality and computational efficiency.
The resulting relaxation has a large number of constraints, and it is difficult to solve using state-of-the-art interior point methods.
Therefore, we design a version of the Peaceman-Rachford splitting method (PRSM) that is able to handle a large number of cutting-planes efficiently. In particular, the PRSM algorithm is adding violated RLT-type inequalities iteratively while using warm-starts.
The computational results show that our bounds for the QMSTP outperform bounds from the literature in quality, as well as in computational time required to obtain them.  Interestingly, this is true even for the bounds that are theoretically stronger than ours, i.e.,  the ones of~\cite{GUIMARAES202046}. We expect that this has to do with the fact that the stronger relaxation of~\cite{GUIMARAES202046} is hard to solve to high precision, whereas our bounds can be computed with high accuracy, resulting in improved bounds.
 Hence, our approach shows significant improvement over other methods from the literature, particularly for larger instances, specifically, for graphs with more than 30 vertices.

\subsection*{Notation}
The set of $n\times n$ real symmetric matrices is denoted by ${\mathcal S}^n$.
The space of symmetric matrices is considered with the trace inner product, which for any $X, Y \in {\mathcal S}^{n}$ is defined as $\langle X,Y \rangle \coloneqq \tr (XY)$.
The associated norm is the Frobenius norm  $\| X\|_F \coloneqq \sqrt{\tr (XX)}$.
The cone of symmetric positive semidefinite matrices  of order $n$ is defined as ${\mathcal S}_+^n \coloneqq \set{X \in  {\mathcal S}^n: X\succeq \mathbf{0}}$.
We order the eigenvalues of $X\in {\mathcal S}^n$ as follows $\lambda_1(X) \leq \cdots \leq  \lambda_n(X)$.  If it is clear from the context to which matrix the eigenvalues relate, we denote eigenvalues by $\lambda_i$.
 The operator $\diag \colon \mathbb{R}^{n \times n} \rightarrow \mathbb{R}^n$ maps a square matrix to a vector consisting of its diagonal elements.
 The adjoint operator of $\diag$ is denoted by $\Diag \colon \mathbb{R}^n \rightarrow \mathbb{R}^{n \times n}$.

We denote by ${\mathbf 1}_n$ the vector of all ones of length $n$, and define ${\mathbf J}_n \coloneqq {\mathbf 1}_n {\mathbf 1}_n^\top$. The indicator vector of $S\subseteq V$ is  denoted by $\indvec{S}.$
The all-zero matrix of order $n$ is denoted by ${\mathbf 0}_n$. We use ${\mathbf I}_n$ to denote the identity matrix of order $n$, while its $i$-th column is given by ${\mathbf u}_i$. In case the dimension of ${\mathbf 1}_n$, ${\mathbf 0}_n$, ${\mathbf J}_n$ and $\mathbf{I}_n$ is clear from the context, we omit the subscript.

We define the $n$-simplex as $\Delta_n \coloneqq \set{x \in \R^p : x \geq \zeros,\ \sum_{i=1}^p x_i = n}$ 
and the capped $n$-simplex as~$\bar{\Delta}_n \coloneqq \set{x \in \R^p : \zeros \leq x \leq \ones,\ \sum_{i=1}^p x_i = n}$.
By~$\mathcal{P}_{\mathcal{M}}$ we denote the projection operator onto the set~$\mathcal{M}$. We use $[n]$ to denote the set of integers $\{1,\ldots,n\}.$ { Throughout the paper, we use the (graph) parameters $\alpha$ and $\beta$, which are defined as $\beta = \algconnPn$ and $\alpha = \frac \beta n$. }

Given a subset $S \subseteq V$ of vertices in a graph $G=(V,E)$,  we denote the set of edges with both endpoints in $S$ by $E(S)\coloneqq \{\{i,j\} \in E~:~i,j \in S \}$
and the cut induced by $S$ by ${\delta(S)} \coloneqq \big\{ \{i,j\} \in E ~ : ~ i \in S, j\notin S \big\}$. 
{The Laplacian matrix of a graph $G$ with $n$ vertices and the adjacency matrix $A$ is defined as $L\coloneqq \Diag(A{\mathbf 1}_n) - A $.}

\section{The Quadratic Minimum Spanning Tree Problem} \label{TheQMST}

In this section, we formally introduce the QMSTP. Let  $G=(V,E)$ be a connected, undirected graph with $n=|V|$ vertices and $m=|E|$ edges. Let $ Q = (q_{ef}) \in {\mathcal S}^{m} $  be a matrix of interaction costs between edges of $G$, where $q_{ee}$ represents the cost of edge $e$.

The QMSTP can be formulated as the following binary quadratic programming problem:
\begin{align*} 
\min\limits_{ x \in \mathcal{T}} ~ \sum_{e \in E} \sum_{f \in E} q_{ef} x_e x_f,
\end{align*}
where $\mathcal{T}$ denotes the set of all spanning trees in $G$. Each spanning tree in $\mathcal{T}$ is represented by its incidence vector of length $m$, and therefore
\begin{align}\label{T}
{\mathcal{T}} \coloneqq \bigg \{ x \in \{0,1\}^{m}  ~: ~ \sum_{e\in E} x_e = n-1, ~ \sum_{
e \in {\delta(S)}} x_e \geq 1, ~\forall S\subsetneq V, ~ S\not= \emptyset \bigg \}.
\end{align} 
The constraints of the type 
\begin{align}\label{constr:cut-set}
\sum_{
e \in {\delta(S)}} x_e \geq 1
\end{align}
are known as
the   {\it cut-set constraints}, and they ensure connectivity of a subgraph from $\mathcal{T}$.
If  $Q$ is a diagonal matrix then the QMSTP reduces to the minimum spanning tree problem that is solvable in polynomial time~\citep{Kruskal1956,Prim1957}.

Let us now fix an ordering for the edges $E=\{ e_1,\ldots, e_m \}$.
For $x\in {\mathcal T}$ define $Y\coloneqq (y_{ef}) \in  {\mathcal S}^m$ such that $y_{ef}=1$ if $x_e=1$ and $x_f=1$, and $y_{ef}=0$ otherwise. 
Then, the  QMSTP can be formulated as the following mixed-integer programming  problem, see~\citep{AssadXu}:
\begin{align}\label{eq:exactFormulation-AssadXu}
\begin{aligned} 
 \min\  & \langle Q,Y\rangle  \\
	\text{s.t. } & \diag(Y)=x, \,\,  Y\ones_m  = (n-1)x &&  \\	
	& \mat 0 \leq Y  \leq \mat J_m, \,\, Y\in {\mathcal S}^m, \,\, x \in \mathcal{T}.
\end{aligned}
	\end{align}
 One can verify that the constraints, in combination with the binarity of $x$, are sufficient to obtain the coupling
between  $Y$ and $x$.  {This result was proven in~\citep{AssadXu}, and this connection between the vector of variables $x$ and 
the matrix variable $Y$ is also exploited in this paper.}
Note that each row in $Y$ is an incidence vector of a tree. 
We refer to  \eqref{eq:exactFormulation-AssadXu} as the {\it cut-set} formulation of the QMSTP.

\section{MISDP formulation for the QMSTP}

\citet{Fiedler1973} defined the algebraic connectivity, $a(G)$, of a graph~$G$ as the second smallest eigenvalue of the Laplacian matrix of the graph. 
It is well-known that the algebraic connectivity is greater than zero if and only if $G$ is a connected graph.
In this section, we will exploit the algebraic connectivity of a tree to derive a compact MISDP formulation of the QMSTP. Moreover, we derive two classes of cutting planes resulting from this MISDP formulation, that will be exploited later on.
\subsection{Compact MISDP formulation for the QMSTP}

We start by showing how a lower bound on the algebraic connectivity of a simple graph can be characterized by a linear matrix inequality.

\begin{proposition}\label{prop:LMI for any graph}
Let $G$ be a simple graph on $n \geq 3$ vertices and  $ L$ be the Laplacian matrix of $G$.  Moreover, let $\widetilde{\beta} \in \R$, $\widetilde{\beta} \geq 0$, and $\widetilde{\alpha} =\frac{\widetilde{\beta}}{n}$. 
	Then  $a(G) \geq \widetilde{\beta}$ if and only if
	$ L +  \widetilde{\alpha} \mat J_n -  \widetilde{\beta} \mat I_n \succeq \zeros$.
\end{proposition}
\begin{proof}
Let $0 = \lambda_1 \leq \lambda_2 \leq \dots \leq \lambda_n$ be the eigenvalues of $ L$, the Laplacian matrix of $G$.
	The eigenvalues of $ L +  \widetilde{\alpha} \mat J$ are $ \widetilde{\beta}$ and $a(G)=\lambda_2 \leq \dots \leq \lambda_n$.
	If~$a(G) \geq  \widetilde{\beta}$, then all eigenvalues of~$ L +  \widetilde{\alpha} \mat J$ are greater or equal than $\widetilde{\beta}$ and therefore $L +  \widetilde{\alpha} \mat J - \beta \mat I \succeq \zeros $. Conversely, if $ L +  \widetilde{\alpha} \mat J - \beta \mat I \succeq \zeros $, then all eigenvalues of~$ L +  \widetilde{\alpha} \mat J $ greater or equal to~$\widetilde{\beta}$ and therefore $a(G) \geq \widetilde{\beta}$.
\end{proof}

We now apply Proposition~\ref{prop:LMI for any graph} to the class of trees.
It is known that the algebraic connectivity for trees with $n \geq 3$ vertices lies in the interval between $\algconnPn$ and $1$, see e.g., \citep{GroneMerris90}.
Here, $\algconnPn$ is the algebraic connectivity of the path graph, and $1$ is the algebraic connectivity of the star graph.
It is also known that a  tree on $n$ vertices has exactly $n-1$ edges. Hence, a tree can be characterized as a connected graph with exactly $n-1$ edges, see also \eqref{T}. 
By using $ \widetilde{\beta} = \algconnPn$ in Proposition~\ref{prop:LMI for any graph}, we obtain the following characterization of spanning trees.

\begin{proposition}\label{prop:treeLMI}
	Let $G$ be a simple graph on $n \geq 3$ vertices with $n-1$ edges.
	Let ${L}$ be its Laplacian matrix and let $\beta = \algconnPn$ and $\alpha = \frac{\beta}{n}$.
	Then, $G$ is a tree if and only if 
	$
		 {L} + \alpha \mat{J}_n - \beta \mat{I}_n \succeq \mat{0}.
	$
\end{proposition}

{Since linear matrix inequalities are tractable constraints in an optimization model,} Proposition~\ref{prop:treeLMI} can be used to derive {an}  MISDP formulation for the QMSTP.
To do so, let us first define the set of  adjacency matrices of induced subgraphs of $G$ with $n$ vertices and $n-1$ edges:
\begin{align}\label{matrA}
\mathcal{F} \coloneqq \left\{ X \in \{0,1\}^{n\times n}\cap \mathcal S ^n ~:~ \langle X, \mat J_n \rangle = 2(n-1),\  x_{ij} = 0 \text{ if }\{i,j\}\notin E \right\}. 
\end{align}
{The set of all adjacency matrices of spanning trees induced by $G$ is contained in $\mathcal{F}$, which we denote by} 
 \begin{align}\label{TM}
\mathcal{T}_M = \mathcal{F} \cap \left \{ X \in \mathcal{S}^n ~:~  \Diag(X {\bf 1}) -  X  + \alpha \mat J - \beta \mat I \succeq \zeros \right \},
 \end{align}
 where $\alpha =\frac \beta n$ and $\beta = \algconnPn$.
 Indeed, there is a bijection between the entries in $\mathcal{T}$, see \eqref{T}, and $\mathcal{T}_M$ given by ${\mathcal B}: {\mathcal T}_M \to {\mathcal T}$, where ${\mathcal B}(X)$ maps $X$ to a column vector containing the entries of X corresponding to $E$ with respect to the fixed
 ordering of the edge set.  We let $\mathcal{B}^* : \mathcal{T} \rightarrow \mathcal{T}_M$ denote its adjoint operator, which maps a column vector $x$ to an adjacency matrix containing the elements of $x$ on the positions of $E$ and zeros elsewhere.

{Now, for a given $X \in \mathcal{T}_M$, we define the lifted matrix $Y = \mathcal{B}(X) \mathcal{B}(X)^\top$ such that $Y_{ef} = 1$ if and only if $X_{ij} = 1$ and $X_{k\ell} = 1$, where $e = \{i,j\}$ and $f = \{k,\ell\}$. To handle the rank-one constraint $Y = \mathcal{B}(X) \mathcal{B}(X)^\top$, 
we exploit the following result on binary PSD matrices.}
\begin{theorem}[\cite{de2023misdp}] \label{Thm:binPSDrank}
Let $Z=\begin{psmallmatrix}
X & x \\
 x^\top  & 1  \\ 
\end{psmallmatrix} \succeq \mathbf{0}$ with $\diag({X}) = {x}$. Then, ${\rank(Z)} = 1$ if and only if~${{X} \in \{0,1\}^{n \times n}}$.
\end{theorem}
{Theorem~\ref{Thm:binPSDrank} implies that we can enforce $Y = \mathcal{B}(X)\mathcal{B}(X)^\top$ by requiring $\begin{psmallmatrix}
Y & \mathcal{B}(X) \\
 \mathcal{B}(X)^\top  & 1  \\ 
\end{psmallmatrix} \succeq \mathbf{0}$, $\diag(Y) = \mathcal{B}(X)$ and $Y \in \{0,1\}^m$.}
{Combining~\Cref{prop:LMI for any graph} with \Cref{Thm:binPSDrank},} we formulate the QMSTP as the following mixed-integer semidefinite program:
\begin{subequations}
  \label{eq:integer-sdp1}
	\begin{align}
	\min\  & \langle Q,Y\rangle \\
	\text{s.t. } &  \diag(Y) = {\mathcal B}(X)  \label{subeq:integer-sdp1:diag} \\	
&  \begin{pmatrix} Y & {\mathcal B}(X) \\
 {\mathcal B}(X)^\top &  1  \end{pmatrix} \succeq \zeros \label{subeq:integer-sdp1:psd}\\
	& \Diag (X \ones ) - X + \alpha \mat J_n - \beta \mat I_n \succeq \zeros \label{subeq:integer-sdp1:lmi2}\\
	& Y\in {\mathcal S}^m, \,\,  X \in \mathcal{F}, \label{subeq:integer-sdp1:Xadjacency}
	\end{align}
 \end{subequations}
 where $\mathcal F$ is given in \eqref{matrA}. 
In  \eqref{eq:integer-sdp1}, we do not impose integrality on the off-diagonal elements of $Y$ as those follow from the integrality of ${\mathcal B}(X)$, see~\citep{de2023misdp}.

The MISDP formulation~\eqref{eq:integer-sdp1} provides an alternative to the mixed-integer linear formulation~\eqref{eq:exactFormulation-AssadXu}. The natural question arises which of these formulations is stronger, i.e., what is the relationship between the continuous SDP (or actually DNN) relaxation of~\eqref{eq:integer-sdp1} and the continuous LP relaxation of~\eqref{eq:exactFormulation-AssadXu}. We now show that none of these relaxations dominates the other. To do so, we consider the natural mappings between solutions of~\eqref{eq:exactFormulation-AssadXu} and~\eqref{eq:integer-sdp1} given by $(x,Y) \mapsto (\mathcal{B}^*(x),Y)$ and $(X,Y) \mapsto (\mathcal{B}(X),Y)$. We show that for each relaxation there exists a feasible solution whose image under this natural mapping is not feasible for the other problem.

\begin{wrapfigure}{R}{0.4\linewidth}
    \centering
    \includegraphics[width=0.8\linewidth]{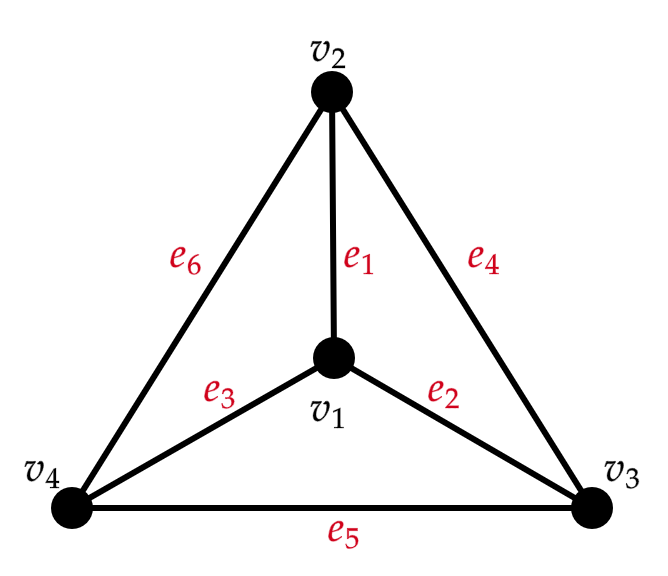}
    \caption{$G = K_4$ with given edge and vertex labeling.\label{Fig:counter1}}
    \vspace{-1cm}
\end{wrapfigure}

We start by showing that the continuous relaxation of~\eqref{eq:exactFormulation-AssadXu} does not dominate the continuous relaxation of~\eqref{eq:integer-sdp1}. Let $G = K_4$ with the edge labeling as presented in Figure~\ref{Fig:counter1} and take 
\begin{align*}
    x & = \begin{pmatrix}
        \frac{2}{3} & \frac{2}{3} & \frac{2}{3} & \frac{1}{3} & \frac{1}{3} & \frac{1}{3} 
    \end{pmatrix}^\top,  \\
    Y & = \begin{pmatrix}
        \frac{1}{3}\mathbf{I}_3 + \frac{1}{3}\mathbf{J}_3 & \frac{2}{3} \mathbf{I}_3 \\
        \frac{2}{3} \mathbf{I}_3 & \frac{1}{3} \mathbf{I}_3
    \end{pmatrix}. 
\end{align*}
It is not difficult to verify that the tuple $(x,Y)$ is feasible for the continuous relaxation of~\eqref{eq:exactFormulation-AssadXu}. Indeed, we have $\diag(Y) = x$, $Y \mathbf{1}_6 = (n-1)x$, $\mathbf{0} \leq Y \leq \mathbf{J}_6$, $Y \in \mathcal{S}^6$ and $x \in \mathcal{T}$, where the last inclusion follows from the fact that $x$ satisfies all cut-set constraints~\eqref{constr:cut-set}.  Now, let us define the matrix $Z$ as $Z=\begin{psmallmatrix}
Y & x \\
 x^\top  & 1  \\ 
\end{psmallmatrix}$.  It turns out that $Z$ has eigenvalue $-\frac{1}{3}$, thus $Z$ is not positive semidefinite. We conclude that $(\mathcal{B}^*(x), Y)$ is not feasible for~\eqref{eq:integer-sdp1}. 

Next, we show that the continuous relaxation of~\eqref{eq:integer-sdp1} does not dominate the continuous relaxation of~\eqref{eq:exactFormulation-AssadXu}. Consider again $G = K_4$ with the labeling as depicted in Figure~\ref{Fig:counter1}. Now, take
\begin{align*}
    X = \begin{pmatrix}
        0 & \frac{1}{4} & \frac{1}{4} & \frac{1}{4} \\
        \frac{1}{4} & 0 & \frac{3}{4} & \frac{3}{4} \\
        \frac{1}{4} & \frac{3}{4} & 0 & \frac{3}{4} \\
        \frac{1}{4} & \frac{3}{4} & \frac{3}{4} & 0 
    \end{pmatrix} \quad \text{and} \quad 
    Y =  \begin{pmatrix}
        \frac{1}{16}\mathbf{J}_3 + \frac{3}{16}\mathbf{I}_3 & \frac{3}{16} \mathbf{J}_3 \\
        \frac{3}{16} \mathbf{J}_3 & \frac{9}{16}\mathbf{J}_3 + \frac{3}{16}\mathbf{I}_3
    \end{pmatrix},
\end{align*}
for which it holds that $\mathcal{B}(X) = \begin{pmatrix}
    \frac{1}{4} & \frac{1}{4} & \frac{1}{4} & \frac{3}{4} &\frac{3}{4} &\frac{3}{4} 
\end{pmatrix}^\top$. One can verify that $\diag(Y) = \mathcal{B}(X)$, $\begin{psmallmatrix}
Y & \mathcal{B}(X) \\
 \mathcal{B}(X)^\top  & 1  \\ 
\end{psmallmatrix} \succeq \mathbf{0}$, $\Diag(X\mathbf{1}_4) - X + \alpha \mathbf{J}_4 - \beta \mathbf{I}_4 \succeq \mathbf{0}$, $Y \in \mathcal{S}^6$ and $X \in \mathcal{F}$. Hence, $(X,Y)$ is feasible for~\eqref{eq:integer-sdp1}. Now, let us define $x = \mathcal{B}(X)$. Then, for $S = \{v_1\}$ we have $\sum_{e \in \delta(\{v_1\})}x_e = \frac{1}{4} + \frac{1}{4} + \frac{1}{4} < 1$, so the cut-set constraint for  $S$ is not satisfied. We conclude that $x \notin \mathcal{T}$, implying that $(\mathcal{B}(X),Y)$ is not feasible for~\eqref{eq:exactFormulation-AssadXu}. 

We summarize these results in the following corollary.

\begin{corollary}
The continuous relaxation of the cut-set QMSTP formulation~\eqref{eq:exactFormulation-AssadXu} and the continuous relaxation of \eqref{eq:integer-sdp1} do not dominate one another.
\end{corollary}

\subsection{Valid inequalities}\label{sec:validinequalities}
In this section, we derive  Chv\'atal-Gomory cuts from the MISDP formulation of the QMSTP introduced in the previous section. 
Interestingly, both the cut-set constraints \eqref{constr:cut-set} and some of the first-level RLT constraints turn out to be CG cuts resulting from~\eqref{eq:integer-sdp1}.

The Chv\'atal-Gomory (CG) cuts form a well-known family of cutting planes for integer programs, initially designed for integer linear programs by~\citet{Chvatal} and~\citet{Gomory}. Similar cuts for integer conic problems and integer semidefinite programs have been investigated by~\citet{CezikIyengar} and~\citet{de2022chv}, respectively.

In the sequel, we derive the cut-set constraints \eqref{constr:cut-set} as CG cuts. 
Let $S \subsetneq V$, $S\neq \emptyset$ and $X$ be feasible for \eqref{eq:integer-sdp1}. Then, 
for the PSD matrix  $\indvec S \indvec{S}^\top$ 
we have that 
\begin{align}\label{LMIinequal}
    \langle \indvec S \indvec{S}^\top, \Diag(X \mat 1) -  X + \alpha \mat J - \beta \mat I \rangle \geq 0,
\end{align}
is a valid inequality for \eqref{eq:integer-sdp1}.  { Here, we exploit the fact that the trace inner product of two positive semidefinite matrices is non-negative.}
After rewriting \eqref{LMIinequal} and exploiting
$\langle \indvec{S}\indvec{S}^\top, \Diag(X {\mathbf 1}) \rangle = \langle  \indvec{S} \mat{1}^\top, X \rangle $, we have 
$
\langle \indvec{S}\indvec{S}^\top - \indvec S \ones^\top, X \rangle \leq \langle \indvec S \indvec{S}^\top, \alpha \mat J - \beta \mat I \rangle .
$
Since the left-hand side of this inequality is integer, we may  round down the right-hand side to the nearest integer value, which results in the CG cut
$    \langle \indvec{S}\indvec{S}^\top - \indvec S \ones^\top, X \rangle \leq \lfloor \langle \indvec S \indvec{S}^\top, \alpha \mat J - \beta \mat I \rangle \rfloor$. 
After rewriting the left-hand side, this inequality becomes
\begin{align}
 - \sum_{i\in S} \sum_{j\notin S} x_{ij} \leq  \left \lfloor \lvert S \rvert (\lvert S \rvert \alpha - \beta) \right \rfloor.    
\end{align}
For $\alpha = \frac \beta n$  and $\beta = \algconnPn$, {see \Cref{prop:treeLMI}}, we have $\lfloor \lvert S \rvert (\lvert S \rvert \alpha - \beta)\rfloor= -1 $, and thus the above CG cut implies the  cut-set constraint  
$ \sum_{e\in {\delta(S)}} x_{e} \geq 1$, see also
\eqref{constr:cut-set}.
Let us summarize the previous discussion.
\begin{proposition}\label{prop:cutConstraints}
Let $S\subsetneq V$, $S\neq \emptyset$. Then, the cut-set constraint \eqref{constr:cut-set} is a Chv\'atal-Gomory  cut with respect to the MISDP \eqref{eq:integer-sdp1}.
\end{proposition}

Subsequently, we derive {other} valid  inequalities by exploiting the constraint \eqref{subeq:integer-sdp1:psd}, that may be equivalently reformulated as 
$Y - {\mathcal B}(X) {\mathcal B}(X)^\top \succeq {\mathbf 0}$.
Let  $X$, $Y$ be feasible for \eqref{eq:integer-sdp1},  $i \in V$, and $\indvec{\delta(i)}$ be the indicator vector of $\delta(i)$.
For $f\in E$, we define the following positive semidefinite matrix
$
P_f \coloneqq {\mathbf u}_k  \indvec{\delta(i)}^\top +\indvec{\delta(i)}{\mathbf u}_k^\top + {\mathbf I}_m +  (n-1) {\mathbf u}_k {\mathbf u}_k^\top,
$
where the index $k$ corresponds to the ordering number of the edge $f$, i.e., ${\mathcal B}(X)_k = x_f$.
Since $P_f \succeq {\mathbf 0}$, it follows that 
$\langle Y - {\mathcal B}(X) {\mathcal B}(X)^\top, P_f  \rangle \geq 0.$
By rewriting the left-hand side, we have 
$$
  \langle Y - {\mathcal B}(X) {\mathcal B}(X)^\top, P_f  \rangle =  2 \sum_{e \in \delta(i)}y_{fe} - 2 x_f  {\mathcal B}(X)\indvec{\delta(i)}^\top \geq 0,
$$
from where it follows 
$\sum_{e \in \delta(i)}y_{fe} \geq x_f  {\mathcal B}(X)\indvec{\delta(i)}^\top \geq x_f,$
since ${\mathcal B}(X)\indvec{\delta(i)}^\top \geq 1$ due to the fact that the underlying graph is connected. 
{
Observe that ${\mathcal B}(X)\indvec{\delta(i)}^\top \geq 1$ is a cut-set constraint, which by \Cref{prop:cutConstraints} is a CG cut. Therefore, the latter inequality implicitly contains a rounding step.}
Thus, we have obtained the constraints 
\begin{align} \label{eq:RLT}
    \sum_{e \in \delta(i)}y_{fe} \geq x_f \qquad \forall f \in E,~\forall i \in V.
\end{align}
Interestingly, these constraints also follow from the reformulation-linearization technique~\citep{sherali2013reformulation}  applied to the cut-set constraints \eqref{constr:cut-set} with $|S| = 1$. 
{ The reformulation-linearization technique (RLT) is a systematic method to derive stronger relaxations for non-convex programming problems by pairwise multiplication of constraints and/or variables, followed by a linearization step~\citep{sherali2013reformulation}.}
Indeed, after multiplying both sides of \eqref{constr:cut-set} by $x_f$ and replacing $x_fx_e$ by $y_{fe}$, one obtains the constraints~\eqref{eq:RLT}.
 Although sometimes written in the form of subtour elimination constraints, the inequalities~\eqref{eq:RLT} have been considered for the QMSTP in the works of~\cite{ONCAN20101762}, \cite{Rostami2015LowerBF}, \cite{PEREIRA2015149}, \cite{GUIMARAES202046}, and~\cite{Sotirov2024}.
We refer later to the constraints~\eqref{eq:RLT} as {\it  RLT-type constraints}.
\begin{proposition}
Let $i \in V$ and $f \in E$.
Then, the constraint \eqref{eq:RLT} is a Chv\'atal-Gomory  cut with respect to the MISDP \eqref{eq:integer-sdp1}.
\end{proposition}
{
In the next section, we show how the inequalities~\eqref{eq:RLT} are exploited to strengthen our  DNN relaxation for the QMSTP, see~\eqref{eq:sdp2}. }

\section{DNN relaxation}
\label{sect:DNN}
Here, we derive two doubly nonnegative  relaxations for the QMSTP and show that these relaxations are not strictly feasible. We apply facial reduction in order to derive facially reduced equivalent formulations.  We emphasize that the relaxations introduced here are not the strongest QMSTP relaxations in the literature. In fact, the SDP relaxation of~\cite{GUIMARAES202046} is theoretically stronger. However, in anticipation of the algorithm for solving our relaxations, see Section~\ref{Sec:PRSM}, we aim at relaxations that are both strong and tractable to solve in order to balance the quality of the bound and the efficiency of computing it.

{Starting from the formulation~\eqref{eq:integer-sdp1}, we introduce a vector~$y$   that is equal to the diagonal of the matrix variable 
 $Y$.} We drop the linear matrix inequality~\eqref{subeq:integer-sdp1:lmi2},
  and relax the constraint  $X \in \mathcal F$, see \eqref{subeq:integer-sdp1:Xadjacency},
  to  $\ones^\top y = n - 1$.
Furthermore, we add the constraint $Y \ones = (n-1) y$ that can be derived from~\eqref{eq:integer-sdp1}.
Additionally, we impose nonnegativity constraints on the matrix variable, and obtain the following DNN relaxation:
\begin{subequations}
  \label{eq:sdp1}
	\begin{align}
	\min\  & \langle Q,  Y \rangle\\
	\text{s.t. }   & \diag(Y) =  y \label{subeq:sdp1:diag}\\ 
	& Y \ones = (n-1)  y, \,\,  \ones^\top y = n-1 \label{subeq:sdp1:sumy}\\
    & {Y \ge \zeros}, ~~ 
         \begin{pmatrix}  Y &  y \\  \mattran y & 1 \end{pmatrix} \succeq \zeros, \label{subeq:sdp1:psd}
        \end{align}
\end{subequations}
{where $Y\in {\mathcal S}^m$.}
The above relaxation does not include the connectivity constraint \eqref{subeq:integer-sdp1:lmi2}, because that constraint has only a small impact on the bound {and} it makes the relaxation more difficult to solve. In order to include a type of connectivity constraints in \eqref{eq:sdp1}, we consider valid inequalities from Section~\ref{sec:validinequalities}.
Preliminary {computational} results show that by adding the cut-set constraints~\eqref{constr:cut-set}, see also Proposition~\ref{prop:cutConstraints}, the resulting bound only marginally improves on the DNN bound~\eqref{eq:sdp1}. The RLT-type cuts~\eqref{eq:RLT}, however, turn out to have a more positive impact on the bound value. We therefore present the following strengthening of the relaxation~\eqref{eq:sdp1}:
	\begin{align} \label{eq:sdp2}
	\begin{aligned}
	\min\  & \langle Q,  Y \rangle\\
	\text{s.t. }   & \text{\eqref{subeq:sdp1:diag}--\eqref{subeq:sdp1:psd}} \\
 & \sum_{e \in \delta(i)}y_{fe} \geq y_f \qquad \forall f \in E,~\forall i \in V.
 \end{aligned}
        \end{align}
The RLT-type cuts~\eqref{eq:RLT} that are added in~\eqref{eq:sdp2} belong to a class of exponentially many RLT-type of subtour elimination constraints that may be used for solving the QMSTP, see e.g.,~\cite{PEREIRA2015149,GUIMARAES202046,Rostami2015LowerBF,Sotirov2024}. Several relaxations including the RLT-type of subtour elimination constraints have been proposed in the literature. In particular, \cite{PEREIRA2015149} propose an incomplete first-level RLT relaxation for the QMSTP that contains a large number of RLT-type of subtour elimination constraints. We emphasize that, although written differently, the cuts~\eqref{eq:RLT} are all included in the relaxation proposed by~\cite{PEREIRA2015149}. \cite{GUIMARAES202046} propose an SDP relaxation for the QMSTP that consists of the incomplete first-level RLT relaxation of~\cite{PEREIRA2015149} accompanied by an additional PSD constraint as in~\eqref{subeq:sdp1:psd}. Hence, both the relaxations~\eqref{eq:sdp1} and~\eqref{eq:sdp2} are theoretically dominated by the SDP relaxation of~\cite{GUIMARAES202046}. \cite{Sotirov2024} compute bounds for the complete first-level RTL relaxation for the QMSTP and show that those bounds   dominate the bounds from~\cite{PEREIRA2015149}.  Finally, \cite{Rostami2015LowerBF} derive an incomplete second-level RLT relaxation, that dominates previously mentioned first-level RLT relaxations, but not the SDP relaxations.

From the previous discussion it becomes clear that~\eqref{eq:sdp1} and~\eqref{eq:sdp2} would benefit from adding the remaining (exponentially many)  RLT-type subtour elimination constraints, which would result
in the stronger relaxation proposed by~\cite{GUIMARAES202046}. We decide not to do so, due to the type of algorithm we will use to solve our relaxations in Section~\ref{Sec:PRSM}. The algorithm of~\cite{GUIMARAES202046} is based on a Lagrangian relaxation scheme where the PSD constraint is dualized, after which the subproblems are solved as minimum spanning tree problems. Thus, the exponentially many subtour elimination constraints are implicitly dealt with by a combinatorial algorithm, e.g., the classical MST algorithm of~\cite{Prim1957}. In contrast, we solve our relaxations with the Peaceman-Rachford splitting method~\citep{peacemanrachford1955} originating from convex optimization. This algorithm is suitable for efficiently solving large-scale SDPs. As a drawback, however, the addition of many polyhedral cutting planes complicates the approach, especially when these cuts have overlapping support, see Section~\ref{subsect:cuttingplane} for details. For that reason, we restrict ourselves to the constraints~\eqref{eq:RLT}, resulting in an efficient algorithm that computes bounds that improve on the basic SDP relaxation~\eqref{eq:sdp1}.
Although the resulting bound is theoretically weaker than the bound from~\cite{GUIMARAES202046}, it is interesting to note that this is not what we observe in practice. Since our approach is able to solve the relaxations up to high precision, the resulting bounds are in practice stronger than the ones reported by~\cite{GUIMARAES202046}.
We refer to Section~\ref{sec:numericalresults} for details.

In the remaining part of this section,  we consider the strict feasibility of the DNN relaxations.
Let
$\Ylarge = \left (\begin{smallmatrix}  Y &  y \\  \mattran y & 1 \end{smallmatrix} \right )$
and $\Qlarge = \left (\begin{smallmatrix}   Q &  \zeros_m \\  \mattran \zeros_m & 0 \end{smallmatrix} \right ). $
{A DNN relaxation is called strictly or Slater feasible if there exists a feasible matrix $\Ylarge \succ \mathbf{0}$ that satisfies all linear inequalities with strict inequality. Such a matrix is called a Slater feasible point.}
It is not difficult to verify that 
\begin{align}\label{eignevectorT}
    T = \begin{pmatrix} \ones_m\\
    -(n -1) \end{pmatrix}
\end{align}
is an eigenvector corresponding to the zero eigenvalue of any matrix $\Ylarge$  feasible for \eqref{eq:sdp1}.
Since there is no feasible matrix~$\Ylarge$ which is positive definite, the DNN relaxation~\eqref{eq:sdp1}
has no Slater feasible point. 

{ As the intersection of the feasible set of~\eqref{eq:sdp1} and the interior of $\mathcal{S}^{m+1}_+$ turns out to be empty, the feasible set of~\eqref{eq:sdp1} is entirely contained in one of the faces of $\mathcal{S}^{m+1}_+$. The facial structure of the PSD cone can be exploited to project the feasible set of~\eqref{eq:sdp1} onto the minimal face of $\mathcal{S}^{m+1}_+$ that contains it~\citep{drusvyatskiy2017many}. This facial reduction technique can be used to obtain equivalent DNN relaxations that are strictly feasible, see~\citep{Hu2023}.}

To provide a facially reduced DNN relaxation  of~\eqref{eq:sdp1},
let $W\in   \R^{(m+1)\times m} $ be a matrix whose columns form a basis for ${\mathcal W}= {\rm null}(T^\top)$,  see~\eqref{eignevectorT}. 
As we will show in Theorem~\ref{thm:strictfeadsible} later on in this section, the relaxation \eqref{eq:sdp1} may  be equivalently written as the following facially reduced relaxation:
\begin{align}  \label{eq:facial}
	\begin{aligned} 
	\min\  & \langle \Qlarge,  W R W^\top \rangle\\
	\text{s.t. } & \diag(  W R W^\top) =   (W R W^\top) \mat u_{m+1} \\
 & (  W R W^\top)_{m+1,m+1} = 1 \\
    & W R W^\top \geq \zeros, \quad  
     R \succeq \zeros.
    \end{aligned}
	\end{align}

We obtained this relaxation from \eqref{eq:sdp1} by replacing $\Ylarge$ with $W R W^\top$ and removing redundant constraints. Note that the feasible set of~\eqref{eq:sdp1} is contained in $W \mathcal{S}^m_+ W^\top$, which is a face of $\mathcal{S}^{m+1}_+$.
To show that~\eqref{eq:facial} has an interior point, we use Theorem~3.15 from~\citep{Hu2023}. 
That theorem additionally takes into account a zero pattern in the feasible matrix, which is not present in our problem. 
\begin{theorem}[Theorem 3.15 in~\cite{Hu2023}]
	\label{thm:slater-facialred-hu}
	Let
	$
		\mathcal{Q} = \set[\Bigg]{y \in \R^{m} ~:~ 
							\mathcal{A}\roundbr[\bigg]{\roundbr[\bigg]{
								\begin{matrix}
									yy^\top & y \\
									y\top & 1
								\end{matrix}
							}} = \zeros,\ y \geq \zeros},
	$
	where $\mathcal{A}$ is a linear transformation, be the feasible set of a quadratically constrained program.
	Suppose~$\aff\roundbr{\conv\roundbr{\mathcal{Q}}} = \mathcal{L}$
	with $\dim\roundbr{\mathcal{L}} = p$. Then, there exist  a matrix $C$ with
	full row rank and $d$ such that
	$		\mathcal{L} = \set[\big]{y \in \R^{m} : Cy = d}.$
	
	Let $M = \roundbr{\begin{matrix} C & -d \end{matrix}}$ and $W$
	be a matrix such that its columns form a basis of ${\rm null}\roundbr{M}$.
    Let~$\mathcal J = \set[\big]{(i,j) : y_{i}y_{j} = 0 \ 
			\forall y \in \mathcal{Q}}$ and $\mathcal J^c$ be its complement.
   Then, there exists a Slater point $\hat R$ for 
   the facially reduced, DNN feasible set:
  	\begin{equation*}
		\hat{\mathcal{Q}}_R = \set[\Big]{R \in \mathcal{S}^{p+1} : R \succeq \zeros,\ 
				\big({WRW^\top}\big)_{\mathcal J} = 0,\ \big({WRW^\top}\big)_{\mathcal J^c} \geq \zeros,\  
				\mathcal{A}\roundbr[\big]{WRW^\top} = \zeros}.
	\end{equation*}
\end{theorem}
We are now ready to state the following result on our facially reduced problem. 
\begin{theorem} \label{thm:strictfeadsible}
    For $n \geq 3$,  the DNN relaxation~\eqref{eq:facial} is a strictly feasible equivalent reformulation of~\eqref{eq:sdp1}.
\end{theorem}
\begin{proof}    
    Let
    \begin{align*}
        \mathcal{Q}_n(m) &=\set[\big]{y \in \set{0,1}^m : \ones_m^\top y = n - 1}
        =\set[\big]{y \in \R^m : \ones_m^\top y = n -1, \ y_i y_i = y_i \ \forall i \in [m]}\\[0.1em]
        &= \set[\bigg]{y \in \R^m : 
                         \mathcal{A}\roundbr[\bigg]{\begin{pmatrix} yy^\top & y \\ y^\top & 1 \end{pmatrix}} = \zeros,
                         \ y \geq \zeros},
    \end{align*}
     where
   $
         \mathcal{A}(X)= \begin{pmatrix}
                        \mathcal{A}_1(X),
                        \cdots ,                      \mathcal{A}_{2m+1}(X)
                        \end{pmatrix}^\top
   $
    with
    \begin{align*}
        \mathcal{A}_{i}(X) &= \bigg\langle \begin{pmatrix}
            \mat u_i  \mat u_i^\top & -\frac 1 2 \mat u_i\\
            -\frac 1 2 \mat u_i^\top & 0
        \end{pmatrix}, X \bigg\rangle &&\text{for all } i \in [m] \text{, and}\\
    \mathcal{A}_{m+i}(X) &=   
    \bigg\langle \frac{1}{2}\bigg(\mat u_i \begin{pmatrix} \ones_m^\top & -(n-1) \end{pmatrix} + \begin{pmatrix}\ones_m \\ -(n-1)\end{pmatrix} u_i^\top \bigg) , X \bigg\rangle
                            &&\text{for all }  i \in [m + 1].
    \end{align*}
    Note that in the definition of~$\mathcal Q_n(m)$, the equality $\mathcal{A}_{i}(X) = 0$ models the constraint $y_i^2 = y_i$
    for all $i \in [m]$.
     The constraint $\mathcal{A}_{2m+1}(X) = 0$ models the constraint $\ones_m^\top y = n -1$.
    For the indices $i \in [m]$, the constraint $\mathcal{A}_{m + i}(X) = 0$ models the redundant constraint~$y_i(\ones_m^\top y) = (n -1) y_i$ for
    all~$1 \leq i \leq m$.
    
    The convex hull equals
    $\conv(\mathcal{Q}_n(m)) =\set[\big]{y \in [0,1]^m : \ones_m^\top y = n - 1}$.
    For each index $i \in [m]$ there exist vectors $y^1, y^2 \in  \mathcal{Q}_n(m)$ such that
    $y^1_i > 0$ and $y^2_i < 1$,
    hence, we get that the affine hull is
$    \aff(\conv( \mathcal{Q}_n(m))) = \set{y \in \R^m : \ones_m^\top y = n - 1}, $    
    and has dimension $m - \rank(\ones_m^\top) = m - 1$. 
 Hence,~$M = T^\top$ where $M$ is from Theorem~\ref{thm:slater-facialred-hu} and $T$ given in \eqref{eignevectorT}.
Let $W \in \R^{(m+1)\times m}$ be a matrix whose columns form a basis of the nullspace of $M$.
 Then, a face of ${\mathcal S}^{m+1}_+$ containing the feasible set of \eqref{eq:sdp1} is of the form $W {\mathcal S}^{m}_+ W^\top$. Therefore, one can replace $\Ylarge$ with $W R W^\top$ in \eqref{eq:sdp1}.

  Moreover,  it holds that for each pair of indices $(i,j) \in [m] \times [m]$,
    there exists a vector $y \in \mathcal{Q}_n(m)$ such that $y_i = y_j = 1$, and hence
    the index set~$\mathcal J = \set[\big]{(i,j) : y_{i}y_{j} = 0 \ \forall y \in \mathcal{Q}}$  is empty.
    Thus, by~\Cref{thm:slater-facialred-hu}, there exists a Slater feasible point for the facially reduced 
    DNN relaxation \eqref{eq:facial}.
        \end{proof}
On top of imposing strict feasibility, facial reduction reduces both the number of variables and constraints. Therefore, the relaxation~\eqref{eq:facial} is preferred over~\eqref{eq:sdp1}. In a similar fashion, relaxation~\eqref{eq:sdp2} can be rewritten by replacing $\Ylarge$ in~\eqref{eq:sdp2} by $WRW^\top$.

\section{Peaceman-Rachford splitting method for the QMSTP} \label{Sec:PRSM}

Interior point solvers have difficulties computing our DNN relaxations for medium-sized problems in a reasonable time due to the large number of (inequality) constraints.
Therefore, we use the Peaceman-Rachford splitting method  (PRSM) for computing the bounds.
The PRSM was first proposed in~\citep{peacemanrachford1955,lionsmercier1979} and is
a symmetric variant of the alternating direction method of multipliers (ADMM). For more details and convergence results, we refer to~\citep{He2016}.

\subsection{PRSM for solving the DNN relaxation}

In this section, we outline the main steps of the  Peaceman-Rachford splitting
method  for solving the DNN relaxation for the QMSTP~\eqref{eq:facial}.

Recall that the matrix $W$ should be such that its columns provide a basis for $\mathcal{W}= {\rm null}(T^\top)$. For reasons explained later, we additionally require the columns of $W$ to be orthonormal. Therefore, we take $W$ as the matrix obtained from applying a QR decomposition to $(\begin{matrix}(n-1)\mat I_m & \ones_m \end{matrix})^\top$.

Now, we define the following sets
\begin{align}
    \mathcal R & \coloneqq \left \{ R \in S^m ~\colon~ R \succeq \zeros,\ \tr(R) = n \right \}, \label{setR}\\
    \mathcal Y & \coloneqq \bigg\{ \Ylarge \in S^{m+1} ~\colon~ \Ylarge = \begin{pmatrix}
        Y & y \\
        y^\top & 1
    \end{pmatrix},\ \diag(Y) = y,\ \zeros \leq \Ylarge \leq \mat J,\ \tr(\Ylarge) = n\bigg\},\label{setY}
\end{align} 
and rewrite~\eqref{eq:facial} as
\begin{equation}
    \label{eq:facialred-minimization-oneline}
    \min~ \Big\{ \big\langle \Qlarge, \Ylarge \big\rangle ~\colon~ \Ylarge = WRW^\top,\ 
                    R \in \mathcal R,\ \Ylarge \in \mathcal Y \Big\}.
\end{equation}  
Note that we added redundant constraints to $\mathcal Y$ and $\mathcal R$,
 where the constraint $\tr(R) = n$ holds, since the columns in $W$ are orthonormalized.
Those redundant constraints help for the efficiency of the algorithm, see e.g.,~\citep{deMeijer2023gpp,oliveira2018admm,Li2021}.

{The class of alternating direction augmented Lagrangian methods, to which the PRSM belongs, aims to minimize the augmented Lagrangian function, see e.g.,~\citep{wen2010alternating,{Boyd2011}}.}
For a fixed penalty parameter~$\revision{\tau} > 0$, the augmented Lagrangian function of~\eqref{eq:facialred-minimization-oneline}
w.r.t.~the constraint~$\Ylarge = WRW^\top$ is
\begin{equation*}
    \mathcal L_{\revision{\tau}}(R,\Ylarge, S) = \big\langle \Qlarge, \Ylarge \big\rangle +
         \big\langle S, \Ylarge - WRW^\top \big\rangle + \frac{\revision{\tau}}{2}
         \big\lVert \Ylarge - WRW^\top \big\rVert^2_F.
\end{equation*}
The basic idea  of the PRSM is to iteratively alternate between optimizing~$\mathcal L_{\revision{\tau}}$
with respect to $R$ and $\Ylarge$ and updating the dual variable~$S$.
The $(k+1)$-th iteration of the PRSM
to minimize the augmented Lagrangian function is
\begin{align*}
    R^{k+1} &= \argmin_{R \in \mathcal R} \mathcal L_{\revision{\tau}}(R, \Ylarge^k, S^k)\\
    S^{\frac{k+1}{2}} &= S^k + \gamma_1 \revision{\tau}(\Ylarge^{k} - WR^{k+1}W^\top)\\
    \Ylarge^{k+1} &= \argmin_{\Ylarge \in \mathcal Y} \mathcal L_\revision{\tau}(R^{k+1}, \Ylarge, S^{\frac{k+1}{2}})\\
    S^{k+1} &= S^{\frac{k+1}{2}} + \gamma_2 \revision{\tau}(\Ylarge^{k+1} - WR^{k+1}W^\top),
\end{align*}
with step lengths $\gamma_1 \in (-1, 1)$ and $\gamma_2 \in \big( 0, \frac{1 + \sqrt 5}{2} \big)$
satisfying $\gamma_1 + \gamma_2 > 0$ and $\lvert\gamma_1\rvert < 1 + \gamma_2 - \gamma_2^2$, see~\citep{He2016}.
The optimization problems occurring in this PRSM scheme can be simplified to projection problems.
Namely, optimizing the augmented Lagrangian over~$\mathcal R$ can be simplified to
\begin{align*}
    R^{k+1} &= \argmin_{R \in \mathcal R}~ \langle S^k, -WRW^\top\rangle +
                \frac{\revision{\tau}}{2} \big\lVert \Ylarge^k - WRW^\top \big\rVert_F 
            = \mathcal P_{\mathcal R} \bigg( W^\top \bigg( \Ylarge^k +
                                    \frac{1}{\revision{\tau}} S^k \bigg) W \bigg),
\end{align*}
where we exploited the fact that the columns of $W$ are orthonormal.
The projection~$\mathcal{P}_{\mathcal{R}}(M)$ of a matrix~$M \in \mathcal{S}^{m}$ onto the set $\mathcal R$ can be computed
by projecting the eigenvalues of~$M$ in the spectral decomposition onto the $n$-simplex $\Delta_n$, see e.g.,~\citep{Li2021}.
In more detail, let~$M = U\Diag({\lambda})U^\top$ be the eigenvalue decomposition of~$M$ with~${\lambda}$
denoting the vector of eigenvalues of~$M$, then 
$\mathcal{P}_\mathcal{R}(M) = U \Diag(\mathcal{P}_{\Delta_n}({\lambda})) U^\top.$
The projection onto the simplex can be performed efficiently. We refer to~\citep{Condat2016} for an overview of algorithms for projecting onto the simplex and their complexities.

Similarly, the optimization problem over the polyhedral set~$\mathcal Y$ can be reformulated as
\begin{align*}
    \Ylarge^{k+1} &= \argmin_{\Ylarge \in \mathcal Y}
                    \big\langle \Qlarge, \Ylarge \big\rangle +
                    \big\langle S^{\frac{k+1}{2}}, \Ylarge \big\rangle +
                    \frac{\revision{\tau}}{2} \big\lVert \Ylarge - WR^{k+1}W^\top \big\rVert_F 
                  = \mathcal{P}_{\mathcal Y}\Big( WR^{k+1}W^\top -
                        \frac{1}{\revision{\tau}} \big( \Qlarge + S^{\frac{k+1}{2}} \big) \Big).
\end{align*}
The projection onto $\mathcal Y$ can then be done in the following way
\begin{equation*}
    \mathcal P_\mathcal{Y}\bigg(\begin{pmatrix} Z & z\\ z^\top & \omega \end{pmatrix} \bigg) =
                \mathcal P_{[0,1]} \Bigg( 
                        \begin{pmatrix}
                            Z - \Diag(\diag(Z)) + v & v\\
                            v^\top & 1
                        \end{pmatrix}\Bigg),
\end{equation*}
where $v = \mathcal{P}_{\bar \Delta(n - 1)}\big(\frac 1 3 \diag(Z) + \frac 2 3 z\big)$ and $\mathcal{P}_{[0,1]}$ denotes the elementwise projection onto the interval~$[0,1]$.

\subsection{PRSM for solving the strengthened DNN relaxation}
\label{subsect:cuttingplane} 

In this subsection, we modify the previously described PRSM algorithm that solves the  relaxation~\eqref{eq:facial}, so that it can handle additional RLT-type constraints.

Let us extend the set $\mathcal{Y}$, see~\eqref{setY}, by adding the RLT-type constraints, yielding
\begin{multline*}
    \mathcal Y_{RLT} = \bigg\{ \Ylarge \in S^{m+1} \colon \Ylarge = \begin{pmatrix}
        Y & y \\
        y^\top & 1
    \end{pmatrix},\ \diag(Y) = y,\ \tr(\Ylarge)=n,\ \zeros \leq \Ylarge \leq \mat J,\\
    \sum_{e \in \delta(i)}y_{fe} \geq y_f \quad \forall f \in E,~\forall i \in V \bigg\}.
\end{multline*}
Thus, the strengthened DNN relaxation \eqref{eq:facialred-minimization-oneline}
is as follows
\begin{equation}
    \label{eq:facialred-minimization-oneline2}
    \min~ \Big\{ \big\langle \Qlarge, \Ylarge \big\rangle ~\colon~ \Ylarge = WRW^\top,\ 
                    R \in \mathcal R,\ \Ylarge \in \mathcal Y_{RLT} \Big\}.
\end{equation}  
The  RLT-type constraints make the projection onto~$\mathcal{Y}_{RLT}$ significantly harder. To the best of our knowledge, there is no closed-form expression for the projection onto $\mathcal Y_{RLT}$.
However, one may write $\mathcal{Y}_{RLT}$  as an intersection of sets that are easier
to project on and then use an algorithm to project onto the intersection of convex sets. 
The cyclic Dykstra's projection algorithm~\citep{dykstraboyle1986} is a suitable algorithm. An overview and analysis of algorithms to project onto the intersection of convex sets can be found in~\citep{BauschkeKoch2013}.

To apply Dykstra's cyclic projection algorithm, let $\mathcal{K}$ denote a coloring of the graph $G$, i.e., $\mathcal{K} = \{K_1, \ldots, K_N\}$ is a partitioning of $V$ into independent sets of $G$. 
We then define the polyhedral sets~$\mathcal{Y}^k$ as
\begin{align*}
    \mathcal Y^k \coloneqq \left\{ \Ylarge \in \mathbb{R}^{(m+1)\times (m+1)} \, :\,\, \Ylarge = \begin{pmatrix}
        Y & y \\
        y^\top & 1
    \end{pmatrix},\ \diag(Y) = y,~\sum_{e \in \delta(i)}y_{fe} \geq y_f \quad \forall f \in E,~\forall i \in K_k  \right\},
\end{align*}
for $k = 1, \ldots, N$.
With this we can now rewrite $\mathcal{Y}_{RLT}$ as
$   \mathcal Y_{RLT} = \mathcal{Y} \cap \left( \bigcap_{k = 1}^N \mathcal{Y}^k \right).$

 The projection onto the sets $\mathcal{Y}^k$ can be performed independently over each row~$f \in E$ of~$Y$ and the corresponding entries of~$y_{f}$ in~$\Ylarge$.
 This allows us to restrict ourselves to projections onto the following type of sets
\begin{align}
    T^k_{f} \coloneqq \left\{ z \in \mathbb{R}^{m+2} \, : \, \, z_f = z_{m+1} = z_{m+2},~\sum_{e \in \delta(i)}z_e \geq z_{f} \quad \forall i \in K_k \right\},
\end{align}
 where the first~$m+1$ entries correspond to the $f$-th row of~$\Ylarge$ and the last entry~$z_{m+2}$
corresponds to~$\Ylarge_{m+1,f}$.
The projection onto~$T_f^k$ can then be computed as presented in the following proposition.
\begin{proposition} \label{Lem:projection_RLT} 
Let $a \in \mathbb{R}^{m+2}$, $f \in E$ and let $\mathcal{K} = \{K_1, \ldots, K_N\}$ denote a coloring of $G$. For each $i \in K_k$, we define $g_{i} \coloneqq  \frac{a_{f} + a_{m+1} + a_{m+2}}{3} - \sum_{e \in \delta(i)}a_{e}$ and sort these values in non-increasing order, i.e., $g_{\sigma(1)} \geq g_{\sigma(2)} \geq \dots \geq g_{\sigma(n_k)}$, where $n_k = |K_k|$ and $\sigma \colon [n_k] \to  K_k$ is an appropriate sorting permutation. For each $p \in [n_k]$, let
\begin{align*}
    \omega(p) \coloneqq \frac{\sum_{j = 1}^p \frac{g_{ \sigma(j)}}{d({\sigma(j)})}}{3 + \sum_{j = 1}^p \frac{1}{d({\sigma(j)})}},
\end{align*}
where $d(\sigma(j))$
denotes the degree of vertex $\sigma(j)$ in $G$.
If $g_{i} \leq 0$ for all $i \in K_k$, then $\mathcal{P}_{T^k_f}(a) = z$, where $z_e = a_e$ for all $e \in E\setminus\{f\}$ and $z_{f} = z_{m+1} = z_{m+2} = \frac{a_f + a_{m+1} + a_{m+2}}{3}$. Otherwise, let $p^*$ denote the largest index $p$ for which $g_{\sigma(p)} > \omega(p)$.
Then, $\mathcal{P}_{T^k_f}(a) = z$, where
\begin{align*}
    z_{e} &= \begin{cases} \frac{a_f + a_{m+1} + a_{m+2}}{3} - \omega(p^*) & \text{if $e \in \{f, m+1, m+2\}$}, \\ 
    a_e + \frac{1}{d(i)}(g_{i} - \omega(p^*)) & \text{if $e \in \delta(i)\setminus\{f\}$ for $i \in K_k \setminus V(f)$ with $\sigma(i) \leq p^*$,} \\
    a_e - \frac{1}{d(i)-1} \sum_{e \in \delta(i)\setminus\{f\}} a_e & \text{if $e \in \delta(i) \setminus \{f\}$ for $i \in K_k \cap V(f)$ with $\sum_{e \in \delta(i)\setminus\{f\}} a_e < 0$},\\
    a_e & \text{otherwise.}
    \end{cases}
\end{align*}
\end{proposition}
\begin{proof}
First, observe that if $g_{\sigma(1)} \leq 0$, then $g_{i} \leq 0$ for all $i \in K_k$. Consequently, the projection of~$a$ onto $T^k_f$ is given by $z$, where $z$ is such that $z_e = a_e$ for all $e \in E\setminus\{f\}$ and $z_f = z_{m+1} = z_{m+2} = \frac{a_f + a_{m+1} + a_{m+2}}{3}$.

If $g_{\sigma(1)} > 0$, then 
$\omega(1) = \frac{\frac{g_{\sigma(1)}}{d(\sigma(1))}}{3 + \frac{1}{d(\sigma(1))}} < \frac{g_{\sigma(1)}}{d(\sigma(1))} \leq g_{\sigma(1)}.$
Hence, the largest index $p$ for which $g_{\sigma(p)} > \omega(p)$, i.e., the index $p^*$, exists. Next, we prove that the projection~$z = \mathcal{P}_{T^k_f}(a)$ is of the described form.

    Using the fact that $z_f = z_{m+1} = z_{m+2}$, the vector $z$ can be obtained as the solution of the following optimization problem, where we restrict to the support of the constraints in $T^k_f$.
    \begin{align} \label{eq:optim_problem}
        \begin{aligned}\min_{z} \quad & \sum_{i \in K_k}\,\sum_{e \in \delta(i)\setminus\{f\}} ||a_e - z_e||_2^2 + ||a_f - z_f||^2_2 + ||a_{m+1} - z_{f}||^2_2 + ||a_{m+2} - z_{f}||^2_2 \\
        \text{s.t.} \quad & \sum_{e \in \delta(i)}z_e \geq z_f \qquad \forall i \in K_k.
        \end{aligned}
    \end{align}
    Let $\lambda_{i}$, $i \in K_k$, denote the dual variables corresponding to the constraints of~\eqref{eq:optim_problem}. We further denote by~$V(f)$ the two vertices in~$G$ adjacent to~$f \in E$. Then, the KKT optimality conditions for~\eqref{eq:optim_problem} are as follows
    {
    \begin{align}
          2(z_e - a_e) - \lambda_{i} &= 0 \qquad \forall e \in \delta(i)\setminus\{f\},~\forall i \in K_k \label{KKT1} \\
          6z_f - 2(a_f + a_{m+1} + a_{m+2}) + \sum_{i \in K_k\setminus V(f)}\lambda_{i} &= 0 \label{KKT2}\\
          \sum_{e \in \delta(i)}z_e &\geq z_f \quad\ \, \forall i \in K_k \label{KKT3}\\
          \lambda_{i}(z_f - \sum_{e \in \delta(i)}z_e) &= 0 \qquad \forall i \in K_k \label{KKT4}\\
          \lambda_{i} &\geq 0 \qquad \forall i \in K_k. \label{KKT5}
    \end{align}
    }%
    It follows from~\eqref{KKT1} and~\eqref{KKT2} that we have
    \begin{equation}\label{KKT6}
    \begin{aligned}
        z_f &= \frac{a_f + a_{m+1} + a_{m+2}}{3} - \frac{1}{6}\sum_{i \in K_k \setminus V(f)}\lambda_{i}, &&  \text{and}\\
        z_e &= a_e + \frac{1}{2}\lambda_{i} && \forall e \in \delta(i)\setminus\{f\},\ \forall i\in K_k.
    \end{aligned}
    \end{equation}
 Suppose $K^* \subseteq K_k$ is the set of vertices for which $\lambda_{i} > 0$ at an optimal solution of~\eqref{eq:optim_problem}. The complementary slackness constraints~\eqref{KKT4} then imply that $z_f = \sum_{e \in \delta(i)}z_e$ for all $i \in K^* \setminus V(f)$ and $\sum_{e \in \delta(i)\setminus\{f\}}z_e = 0$ for $i \in K^* \cap V(f)$. Note that $\lvert K^* \cap V(f) \rvert \leq 1$ since $K_k$ is an independent set in~$G$. By exploiting~\eqref{KKT6} and $\sum_{j \in K_k\setminus V(f)}\lambda_{j} = \sum_{j \in K^*\setminus V(f)}\lambda_{j}$, these equations can be rewritten to 
 \begin{align}
     \frac{a_f + a_{m+1} + a_{m+2}}{3} &- \frac{1}{6}\sum_{j \in K^*\setminus V(f)}\lambda_{j} = \sum_{e \in \delta(i)} \left( a_e + \frac{1}{2}\lambda_{i} \right) \nonumber \\
     \Longleftrightarrow \quad \lambda_{i} & = \frac{2}{d(i)} \left(\frac{a_f + a_{m+1} + a_{m+2}}{3} - \sum_{e \in \delta(i)}a_e - \frac{1}{6}\sum_{j \in K^* \setminus V(f)}\lambda_{j} \right) \nonumber \\
     \Longleftrightarrow \quad \lambda_{i} & = \frac{2}{d(i)} \left(g_{i} - \frac{1}{6}\sum_{j \in K^*\setminus V(f)}\lambda_{j} \right) \label{eq:lambdaexpr}
 \end{align}
 for all $i \in K^*\setminus V(f)$. Summing the latter equations over all $i \in K^*\setminus V(f)$ yields
 \begin{equation*}
     \sum_{i \in K^*\setminus V(f)}\lambda_{i} = 2 \sum_{i \in K^*\setminus V(f)} \frac{g_{i}}{d(i)} - \frac{1}{3}\sum_{i \in K^*\setminus V(f)} \frac{1}{d(i)} \sum_{j \in K^*\setminus V(f)}\lambda_{j},
\end{equation*}
or equivalently,
$     \sum_{i \in K^*\setminus V(f)} \lambda_{i} = \frac{2 \sum_{i \in K^*\setminus V(f)}\frac{g_{i}}{d(i)}}{1 + \frac{1}{3}\sum_{i \in K^*\setminus V(f)}\frac{1}{d(i)}}.$
 After substitution into~\eqref{eq:lambdaexpr}, we obtain 
 \begin{align} \label{KKT7}
     \lambda_{i} = \frac{2}{d(i)} \left(g_{i} - \frac{\sum_{i \in K^*\setminus V(f)}\frac{g_i}{d(i)}}{3 + \sum_{i \in K^*\setminus V(f)}\frac{1}{d(i)}} \right) > 0
 \end{align}
 for all $i \in K^*\setminus V(f)$. For each $i \in (K_k \setminus K^*) \setminus V(F)$, we have $\lambda_i = 0$. The inequalities~\eqref{KKT3} for these $i$ then read
 \begin{equation*} 
     \sum_{e \in \delta(i)}a_e \geq \frac{a_f + a_{m+1} + a_{m+2}}{3} - \frac{\sum_{i \in K^*\setminus V(f)}\frac{g_i}{d(i)}}{3 + \sum_{i \in K^*\setminus V(f)}\frac{1}{d(i)}},
 \end{equation*}
or equivalently, 
\begin{equation}\label{KKT8}
g_i - \frac{\sum_{i \in K^*\setminus V(f)}\frac{g_i}{d(i)}}{3 + \sum_{i \in K^*\setminus V(f)}\frac{1}{d(i)}} \leq 0.
 \end{equation}
 By combining~\eqref{KKT7} and~\eqref{KKT8} we obtain the following optimality conditions on the dual variables $\lambda$ concerning the indices in $K_k \setminus V(F)$
 \begin{align} \label{KKT9}
 \left\{
     \begin{aligned}
         \frac{2}{d(i)} \left(g_i - \frac{\sum_{i \in K^*\setminus V(f)}\frac{g_i}{d(i)}}{3 + \sum_{i \in K^*\setminus V(f)}\frac{1}{d(i)}} \right) & > 0 & & \text{for all $i \in K^*\setminus V(f)$,} \\ 
         g_i - \frac{\sum_{i \in K^*\setminus V(f)}\frac{g_i}{d(i)}}{3 + \sum_{i \in K^*\setminus V(f)}\frac{1}{d(i)}} & \leq 0 & & \text{for all $i \in (K_k \setminus K^*) \setminus V(F)$.}
     \end{aligned} \right.
 \end{align} 
  We conclude from the conditions~\eqref{KKT9} that the support of $\lambda$ restricted to $K_k \setminus V(f)$ always consists of the vertices for which $g_i$ lies above a certain threshold value. To find this threshold value, we sort the $g_i$'s in non-increasing order and check all possible candidate sets for $K^*\setminus V(f)$ corresponding to the first $r$ entries in this sorted list. Let $\sigma \colon [n_k] \to K_k$ denote an according sorting permutation, i.e., $\sigma$ is bijective and fulfills $g_{\sigma(1)} \geq g_{\sigma(2)} \geq \dots \geq g_{\sigma(n_k)}$.
  For each candidate set $\{\sigma(1), \ldots, \sigma(p)\} \subseteq K_k \setminus V(f)$, it suffices to check whether $g_{\sigma(p)}$ is strictly larger than the candidate threshold value
  \[\omega(p) \coloneqq \frac{\sum_{j = 1}^p\frac{g_{\sigma(j)}}{d(\sigma(j))}}{3 + \sum_{j = 1}^p\frac{1}{d(\sigma(j))}}.\] If $p^*$ is the largest index for which this holds, then this candidate set equals $K^*\setminus V(f)$. The existence of such a $p^*$ is guaranteed by the existence of a solution to the projection problem~\eqref{eq:optim_problem}.

 Finally, we need to address the optimality conditions for all $i \in K_k \cap V(f)$.
In case~$i \in K^* \cap V(f)$, we have $\lambda_i > 0$, and due to complementary slackness~\eqref{KKT4} it holds that
$$
	0 = \sum_{e \in \delta(i)\setminus\{f\}} z_e = \sum_{e \in \delta(i)\setminus\{f\}} \Big(a_e + \frac 1 2 \lambda_i \Big),~~\text{or equivalently, }~~\lambda_i = -\frac{2}{d(i)-1} \sum_{e \in \delta(i)\setminus\{f\}} a_e > 0.
$$
We note here that we may w.l.o.g.~assume that $d(i) >  1$. Namely, if $d(i) = 1$, then the set $\delta(i) \setminus \{f\}$ is empty, hence $\lambda_i$ will not appear anywhere in~\eqref{KKT6}, making this dual variable redundant.

For the case $i \in (K_k \setminus K^*) \cap V(f)$, and hence $\lambda_i = 0$, condition~\eqref{KKT4}
with~\eqref{KKT3} reads as 
$\sum_{e \in \delta(i)\setminus\{f\}} a_e \geq 0$. Combining both cases, we obtain the following optimality conditions for $i \in K_k \cap V(F)$:
$$
\left\{
\begin{aligned}
 \sum_{e \in \delta(i)\setminus\{f\}} a_e & < 0 && \text{for $i \in K^* \cap V(f)$,} \\ 
\sum_{e \in \delta(i)\setminus\{f\}} a_e& \geq 0 && \text{for $i \in (K_k \setminus K^*) \cap V(f)$.}
\end{aligned} \right.
$$
Altogether, the equations~\eqref{KKT6} then imply
 \begin{align*}
     z_e &= \begin{cases} \frac{a_f + a_{m+1} + a_{m+2}}{3} - \omega(p^*) & \text{if $f \in \{e, m+1, m+2\}$,} \\
         a_e + \frac{1}{d(i)}(g_i - \omega(p^*)) & \text{if $e \in \delta(i) \setminus \{f\}$ and $i \in K^*\setminus V(f)$,} \\
         a_e - \frac{1}{d(i)-1} \sum_{e \in \delta(i)\setminus\{f\}} a_e & \text{if $e \in \delta(i) \setminus \{f\}$ and $i\in K_k \cap V(f)$ with $\sum_{e \in \delta(i)\setminus\{f\}} a_e < 0,$}\\
         a_e & \text{otherwise.}
     \end{cases}
 \end{align*}
\end{proof}

 It follows from Lemma~\ref{Lem:projection_RLT} that the projection onto $T^k_f$ involves both a sorting and an enumeration of a list of $n_k$ elements. Hence, the worst-case time complexity is $O(n_k \log n_k)$.
 
In fact, for computational purposes, we are not going to project on~$\mathcal{Y}_{RLT}$ but iteratively add
violated cuts only. For that, we denote by~$\mathcal C \subseteq V \times E$ the set of violated
cuts that we to add to~$\mathcal Y$, where an element~$(i,f)$ represents the cut
$\sum_{e \in \delta(i)}y_{ef} \geq y_f$.
We further define analogously to~$\mathcal{Y}_{RLT}$ the polyhedral set
\[\mathcal Y_{\mathcal C} \coloneqq
\left\{ \Ylarge \in \mathbb{R}^{(m+1)\times (m+1)} \, :\,\, \Ylarge = \begin{pmatrix}
        Y & y \\
        y^\top & 1
    \end{pmatrix},\ \diag(Y) = y,~\sum_{e \in \delta(i)}y_{fe} \geq y_f \quad \forall (i,f) \in \mathcal C  \right\}.\]
The projection follows the same idea as explained above for the projection onto~$\mathcal Y_{RLT}$, but in this case, instead of partitioning the vertex set~$V$ into independent sets, we can partition the
constraints in~$\mathcal C$ for each edge~$f$ separately.
For a fixed~$f$, we partition the vertices occurring together with~$f$ in~$\mathcal C$
into independent sets~$K^f_1,\dots,K^f_{N_f}$.
Note that the number of independent sets~$N_f$ for an edge will probably be way smaller
than the number of colors needed to color the whole graph, which can, in the worst case of a complete graph, be the number of vertices.
Furthermore, as mentioned above, it is possible to project independently over each row~$f \in E$, which allows us to parallelize this step.
Hence, we cluster the cut constraints in
$
    \mathcal C_k = \big\{ (i,f) \in {\mathcal C} :  f \in E,\ i \in K^f_k \big\}
$
for $1 \leq k \leq N_{max}$ with $ N_{max} \coloneqq \max\{N_f : f \in E\}$ and obtain
$\mathcal Y_{\mathcal C} = \mathcal Y \cap \Bigg(\bigcap_{k = 1}^{N_{max}}\mathcal Y_{\mathcal C_k}\Bigg),$
where we can easily project onto~$\mathcal Y_{\mathcal C_k}$ using~\Cref{Lem:projection_RLT}.
A pseudocode for the Cyclic Dykstra projection algorithm to project onto~$\mathcal Y_{\mathcal C}$ can be found in~\Cref{alg:pseudocode_dykstra}.
\begin{algorithm}
    \caption{Dykstra's cyclic projection algorithm to project onto~$\mathcal{Y}_{\mathcal C}$}\label{alg:pseudocode_dykstra}
    \footnotesize
    \vspace*{0.2em}
    \textbf{Input:} matrix $M$, cuts $\mathcal C$, $\varepsilon_{proj}$\\
    \textbf{Output:} the projection $\mathcal{P}_{\mathcal{Y}_{\mathcal C}}(M)$ of $M$ onto $\mathcal{Y}_{\mathcal C}$
    \vspace*{0.2em}
    \begin{algorithmic}[1]
        \State cluster $\mathcal C$ into $\{ \mathcal C_{1}, \dots, \mathcal C_{N_{max}}\}$
        \State initialize $X = M$, $P = \zeros$, $Q_1 = \dots = Q_{N_{max}} = \zeros$
        \Repeat
        \State $X_{old} = X$
        \State $X_{tmp} = X + P$
        \State $X = \mathcal P_{\mathcal Y}(X_{tmp})$
        \State $P = X_{tmp} - X$
        \For{$k = 1, \dots, N_{max}$}
            \State $X_{tmp} = X + Q_{k}$
            \State $X = \mathcal{P}_{\mathcal{Y}_{\mathcal C_k}}(X_{tmp})$
            \State $Q_k = X_{tmp} - X$
        \EndFor
        \Until{$\lVert X_{old} - X \rVert < \varepsilon_{proj}$}
        \State \Return $X$
    \end{algorithmic}
\end{algorithm}
To compute the lower bound~\eqref{eq:facialred-minimization-oneline2} with a PRSM algorithm, we first compute the DNN bound~\eqref{eq:facialred-minimization-oneline}
with the PRSM, as explained in the previous subsection. Then, we separate violated cuts from the current solution and add the~\texttt{ncutsmax} most violated ones to~$\mathcal C$. We then proceed to compute~\eqref{eq:facialred-minimization-oneline}
with the additional new cuts in~$\mathcal C$ with the PRSM and use the solution from before for a warm-start.
This process of separating and adding new cuts to~$\mathcal C$ in an outer loop is iterated
until one of the stopping criteria is met. \Cref{alg:prsm} provides a pseudocode for the described algorithm.
\begin{algorithm}[h]
\footnotesize
    \caption{PRSM algorithm to compute lower bounds on the QMST}\label{alg:prsm}
    \vspace*{0.2em}
    \textbf{Input:} graph~$G=(V,E)$, cost matrix~$\Qlarge$\\
    \textbf{Output:} (valid) lower bound LB
    \vspace*{0.2em}
    \begin{algorithmic}[1]
        \State initialize $\Ylarge^0$, $S^0$, $\revision{\tau}$, $\gamma_1$, $\gamma_2$, set $\mathcal C = \emptyset$\Comment{cf.~\Cref{sec:numericalresults}}
        \State compute $W$, e.g., apply QR decomposition to $(\begin{matrix} (n-1)\mat I_m & \ones_m \  \end{matrix})^\top$
        \State $k = 0$
        \While{no stopping criteria met}
            \While{no stopping criteria met}
                \State $R^{k+1} = \mathcal P_{\mathcal R}(W^\top(\Ylarge^k + \frac{1}{\revision{\tau}} S^k  )W) $
                \State $S^{\frac{k+1}{2}} = S^k + \gamma_1 \revision{\tau}(\Ylarge^k - WR^{k+1}W^\top)$
                \State $\Ylarge^{k+1} = \mathcal{P}_{\mathcal Y_\mathcal{C}}\big(WR^{k+1}W^\top - \frac{1}{\revision{\tau}}\big(\Qlarge + S^{\frac{k+1}{2}}\big)\big)$
                \State $S^{k+1} =  S^{\frac{k+1}{2}} + \gamma_2\revision{\tau} (\Ylarge^{k+1} - WR^{k+1}W^\top)$
                \State $k = k + 1$
            \EndWhile
            \State compute a valid lower bound~LB from~$S^k$\Comment{cf.~\Cref{subseq:stoppingcrit-validlb}}
            \State separate violated cuts and add the \texttt{ncutsmax} most violated ones to~$\mathcal C$
            \State cluster the cuts in~$\mathcal C$
        \EndWhile
        \State\Return LB
    \end{algorithmic}
\end{algorithm}

\subsection{Stopping criteria and post-processing of the PRSM algorithm}\label{subseq:stoppingcrit-validlb}

In this subsection, we briefly discuss the stopping criteria and the post-processing phase of our
PRSM algorithm.
\paragraph{Stopping criteria} 
We use several criteria to decide when to stop the inner and outer iterations
of \Cref{alg:prsm}.
The main stopping criteria for the inner while loop is when the primal and dual errors satisfy
\[
    \max \Bigg\{  \frac{\big\lVert \Ylarge^{k+1} - WR^{k+1}W^\top \big\rVert_F}{1 + \big\lVert \Ylarge^{k+1} \big\rVert_F},\ \revision{\tau} \frac{\big\lVert W^\top\big(\Ylarge^{k} - \Ylarge^{k+1}\big)W\big\rVert_F}{1 + \big\lVert S^{k+1}\big\rVert_F} \Bigg\} \leq \varepsilon_{PRSM},
\]
cf.~\citep{Boyd2011}.
We further stop the inner iterations when the maximum number of total PRSM iterations or a time limit
is reached. In that case, we compute a valid dual bound as described below,  and stop the algorithm.

For the outer loop, we have the following possible stopping criteria. 
If an upper bound is known, the algorithm stops as soon as the obtained valid lower bound closes the gap.
We further stop the algorithm if the number of new violated cuts found is below a certain threshold~\texttt{ncutsmin}. If the improvement of the valid lower bound compared to the valid lower bound of the previous outer iteration is smaller than~\texttt{epslbimprov}, we stop the algorithm as well.
And finally, we stop after a maximum of~\texttt{noutermax} outer iterations.

\paragraph{Valid lower bound} 
The value obtained as an output of \Cref{alg:prsm} does not necessarily provide a lower bound for the problem, as the convergence of the  PRSM is typically not monotonic, and one stops the algorithm earlier. Therefore, it is necessary to perform a postprocessing procedure to obtain a valid lower bound.
We apply the approach presented in~\citep{Li2021}.
 The safe lower bound derived by this method is then given by
\[\text{lb}(S^{\text{out}}) =
	\min_{\Ylarge \in \mathcal Y_{\mathcal C}}\langle \Qlarge + S^{\text{out}}, \Ylarge \rangle
	- n \lambda_{\max}(W^\top S^{\text{out}} W),\]
 where $S^{\text{out}}$ denotes the dual matrix variable resulting from (an early stop of) the PRSM. The computation of this lower bound boils down to computing the largest eigenvalue and solving a linear program.
Similarly, one can obtain a valid lower bound from the PRSM algorithm that solves \eqref{eq:facialred-minimization-oneline}, by replacing $\mathcal Y_{\mathcal C}$ with $\mathcal Y$, see \eqref{setY}, in the above expression. 

\section{{Computational} results}\label{sec:numericalresults}

We implemented\footnote{The code can be found on \url{https://github.com/melaniesi/QMST.jl} and as ancillary files on the arXiv page of this paper.} our algorithm in Julia~\citep{bezanson2017julia} version~1.10.0.
For solving the linear program to compute a valid lower bound, we are using the solver HiGHS~\citep{highs} with the modeling language JuMP~\citep{Lubin2023}.
The projection onto~$\mathcal C_k$ is multithreaded\revision{, as described in Section~\ref{subsect:cuttingplane}}.
\revision{For computing the eigendecomposition to perform the projection onto $\mathcal R$, see~\eqref{setR}, we use the Julia function \texttt{eigen}. As we use the Intel Math Kernel Library, this function is executed multithreaded.}
All computations were carried out on an AMD EPYC 7343 with 16 cores with 4.00GHz and 1024GB RAM,
operated under Debian GNU/Linux~11.
\paragraph{Parameter setting}
We initialize the matrices, penalty parameters, and step lengths as follows.
 As starting values for the matrices, we choose  $S^0 = \zeros$ and
\[  \Ylarge^0 = \left (\begin{smallmatrix}
            \frac{(n-1)}{m} \mat I + \frac{(n-1)(n-2)}{m(m-1)}(\mat J - \mat I) & ~\frac{(n-1)}{m} \ones\\[1.5ex]
            \frac{(n-1)}{m} \ones^\top & 1
              \end{smallmatrix} \right )\]
Based on the results of {computational} tests, we have determined the values for the
penalty parameter~$\revision{\tau}$ and step lengths. We set the step length parameters to $\gamma_1 = 0.9$, $\gamma_2 = 1$.
For the penalty parameter, let $q_{\max} \coloneqq \max \set{\tr(Q), \lVert Q \rVert_F}$ and 
$q_{\min} \coloneqq \min \set{\tr(Q), \lVert Q \rVert_F}$, we then set
\[\revision{\tau} = \begin{cases}
    \sqrt{\frac{q_{\min}}{m+1} \lVert Q \rVert_F} & \text{if } \frac{q_{\max}}{q_{\min}} < 1.2, \\
    \sqrt{\frac{q_{\max}}{q_{\min}} \lVert Q \rVert_F} & \text{else.}
\end{cases}\]
We run our algorithm for all instances with~$\varepsilon_{PRSM} = 10^{-4}$ and the parameter~$\varepsilon_{proj}$ is set to~$10^{-5}$.
Violated cuts are considered if the violation is greater than~$10^{-3}$ and after each outer iteration,
the~\texttt{ncutsmax =~$m$} most violated cuts are added. 
No further cuts are added if the improvement of the lower bound is smaller than~\texttt{epslbimprov =~$10^{-3}$} or the number of new violated cuts found is less than~\texttt{ncutsmin = 10}.
The maximum wall-clock time for running our algorithm is set to~3~hours per instance, and the
maximum number of total iterations is set to~$10\,000$.
We set the number of maximum outer iterations to~\texttt{noutermax = 10}.

\paragraph{Benchmark instances}
We test our algorithm on the following three benchmark sets.
The first benchmark set OP was introduced by~\citet{ONCAN20101762}.
The benchmark set consists of 3 different classes, each consisting of 160 instances on complete graphs: the OPsym, OPvsym and OPesym instances.
The OPsym instances have diagonal entries chosen uniformly from~$[100]$, and the off-diagonal values are uniformly distributed at random in~$[20]$.
For instances in the class OPvsym, the diagonal values are uniformly distributed in~$[10\,000]$, and the off-diagonal values $Q_{\{i,j\},\{k,l\}}$ are computed as $w(i) w(j) w(k) w(l)$, where $w \colon V \to [10]$ assigns to each vertex in the graph a uniformly distributed weight at random in~$[10]$. The cost matrix for instances of the
type OPesym is constructed in the following way. First, the vertex coordinates are randomly chosen in the box~$[0,100]\times [0,100]$, and the edges are represented as straight lines connecting vertices. The edge cost~$Q_{ee}$ is then set as the length of the edge~$e$, and the interaction cost between two edges~$e$ and~$f$ is computed as the
Euclidean distance between the midpoints of~$e$ and~$f$.
For each of those test sets, they randomly generated 10 instances each for $n \in \{6, 7, \dots, 17, 18\} \cup \{20, 30, 50\}$.  We do not include the benchmark instances of type OPesym and~$n=20$ in our study, as we were unable to locate the correct instances\footnote{
In the benchmark set~\url{https://data.mendeley.com/datasets/cmnh9xc6wb/1}, the instances indicated as type OPesym for $n=20$ are the OPvsym  for~$n=6$.}.

The second family of benchmark instances CP was introduced by~\citet{cordone2008}. The benchmark set consists of 108 instances divided into~4 classes, specifying the sets from which the diagonal and off-diagonal values of the cost matrix are chosen uniformly at random. For each pair of the number of vertices~$n  \in \{10,15,20,25,30,35,40,45,50\}$, density~$d \in \{33\%, 67\%, 100\%\}$  and class,
one random graph was generated. The values of the cost matrix~$Q$ are uniformly distributed on the sets as listed below.
\begin{table}[H]
    \centering
    \footnotesize
    \begin{tabular}{l|cccc}
        class & CP1 & CP2 & CP3 & CP4\\\hline
        diagonal values & [10] & [10] & [100] & [100]\\
        off-diagonal values & [10] & [100] & [10] & [100]
    \end{tabular}
\end{table}

The last benchmark set SV was introduced by Sotirov and Verch\'ere in their recent paper~\citep{Sotirov2024}.
It consists of~24 instances, with one random graph for each pair of $n \in \{10,12,14,16,18,20,25,30\}$ and $d \in \{33\%, 67\%, 100\%\}$.
They constructed the cost matrices in such a way that
for a given maximum cost for the diagonal entries,
and a maximum cost for the off-diagonal entries, 10\% of
the edges have high interaction costs with each other (between 90 and 100\%
of the maximum off-diagonal cost)  and low interaction costs with the rest (between 20 and 40\% of the maximum
off-diagonal cost). The other 90\% of edges have an interaction cost of between 50 and 70\% of the maximum off-diagonal cost with each other. The diagonal entries are chosen to be between 0 and 20\% of the maximum diagonal cost.

\paragraph{Bounds from the literature} 
We compare our {computational} results to lower bounds from \citep{GUIMARAES202046,PEREIRA2015149,Sotirov2024}. The upper bounds \revision{on the SV instances are taken from~\cite{Sotirov2024}, for all other instances, the upper bounds are taken from a collection of best known results in the literature\footnote{\url{https://homes.di.unimi.it/cordone/research/qmst.html}}.}

The bounds from~\citep{GUIMARAES202046}, called  LAGN and LAGP, are used in the to-date
best exact algorithm for the QMSTP.
Those bounds are obtained from two different ways of dualizing
an SDP relaxation of QMSTP. For LAGN, the semidefiniteness constraint
is dualized, and a subgradient method is used to compute the optimum.
Whereas for computing LAGP, there is no semidefiniteness constraint
present, but a semi-infinite reformulation together with polyhedral
cutting planes is solved using a bundle method.

The lower bounds VS1 and VS2 were introduced by~\citet{Sotirov2024}.
These lower bounds are based on an extended formulation of the minimum quadratic
spanning tree problem and are strengthened by facet defining inequalities of the
Boolean Quadric polytope. The lower bound  VS2 is stronger than VS1. 

\citet{PEREIRA2015149} solved several benchmark problems of sizes
up to~50~vertices  using a RLT based relaxation RLT1.
RLT1 is an incomplete first-level RLT relaxation and is computed by
dualizing the symmetry constraint, applying the GL procedure, and using a subgradient algorithm.
Another RLT based bound among the strongest relaxations in the literature is RLT2, presented in~\citep{Rostami2015LowerBF}. \citet{Rostami2015LowerBF} use a dual-ascent procedure for computing
their relaxation based on the second-level of RLT.

\paragraph{Computational results}
We first present a comparison to the results from~\citep{GUIMARAES202046}, where the authors also compute SDP bounds.
Their computations were carried out on a machine with 32 GB RAM and two E5645 Intel Xeon processors, with six 2.40GHz cores each.

 The structure of \Cref{tab:comparison-bbn} is analogous to Table~4 in~\citep{GUIMARAES202046} and reads as follows. The rows are grouped into~3~blocks, each reporting the results averaged over all CP instances with the same property as specified in the first column of the table.
The first block of rows averages over instances of the same size, the second averages the results over the densities of the graphs, and the last block averages over the different classes of the CP instances.
In the second column of~\Cref{tab:comparison-bbn}, we report the average gap obtained by the valid lower bound  obtained with our PRSM algorithm when stopping after the first outer iteration, cf.~\eqref{eq:facialred-minimization-oneline}.
We compute the relative gap between that lower bound ($\text{LB}_{DNN}$) and the best known upper bound (UB) from the literature
using $100(UB- \text{LB}_{DNN} )/UB$.
 We remark here that the same gap was calculated by~\citet{GUIMARAES202046}.\footnote{There was a typo in that paper that claims differently, but our statement can be easily verified.}
In the third column, we report the average wall clock time in seconds needed to compute this lower bound.
In column~4, we report the average gap obtained by the bound returned by~\Cref{alg:prsm}, cf.~\eqref{eq:facialred-minimization-oneline2}, and in column~5, the average time needed to compute this bound.
In the sixth and seventh column, we list the average gaps and computation times for the bound LAGN of~\citep{GUIMARAES202046}, which is used in the best up-to-date exact algorithm for the QMSTP.
 The average gaps and computation times of  LAGP, the second lower bound introduced in~\citep{GUIMARAES202046},
are given in the last two columns of \Cref{tab:comparison-bbn}.

The results in \Cref{tab:comparison-bbn} show that for the CP instances, our lower bounds 
are, on average, significantly stronger than the SDP bounds LAGN and LAGP.
Except for the instances with $n \in \{10,15\}$, the average computation times for solving our  relaxations are smaller than those reported for computing SDP bounds LAGN and LAGP. The average time to compute the DNN + CUTS bound, that is~\eqref{eq:facialred-minimization-oneline2}, over all CP instances is~$51$~seconds, compared to $1\,360$ and $5\,652$ seconds for LAGN and LAGP, respectively. More significant difference in the computation times and relative gaps can be seen for larger instances.
\revision{Even if the machine used for this study is stronger than the one used for computing LAGN and LAGP, it is evident that with an increasing size $n$, the time for computing LAGN and LAGP grows rapidly.}
One can also observe that the less dense the instances are, the smaller the average relative gap.
Furthermore, the effect of adding cuts is more significant for sparse graphs than for dense graphs.
~\citet[Table 4]{GUIMARAES202046} compare their bounds to RLT1~\citep{PEREIRA2015149},
which can be computed approximately three times faster than LAGN but yields much weaker bounds.
The average gap of bound~RLT2~\citep{Rostami2015LowerBF} over all instances of size~$n \leq 35$ for each of the four CP classes is at least three times larger than our reported average gaps for~\eqref{eq:facialred-minimization-oneline}. 
Overall, \Cref{tab:comparison-bbn} shows that, especially for larger CP instances, our bounds are significantly stronger and faster to compute than any other bounds.

The latter conclusion at first might seem counterintuitive, as our relaxations~\eqref{eq:sdp1} and~\eqref{eq:sdp2} are theoretically weaker than the bound from~\cite{GUIMARAES202046}, where the algorithms LAGN and LAGP are based on. However, it follows from Table~\ref{tab:comparison-bbn} that this is not the case in practice. We believe that the difference might be caused by the fact that the algorithms LAGN and LAGP are stopped before reaching their theoretical optimal solution. This hypothesis is strengthened by the fact that the reported gaps of LAGN and LAGP are different, although they result from the same theoretical SDP bound.
 Since our approach is able to solve the relaxation up to high precision, our bounds turn out to be stronger than the ones reported by~\cite{GUIMARAES202046}. Finally, due to our postprocessing step explained in Section~\ref{subseq:stoppingcrit-validlb}, we have full certainty that our approach provides a valid lower bound even if the algorithm is stopped at a low precision.
\medskip

In the \Cref{tab:CP1,tab:CP2,tab:CP3,tab:CP4,tab:SV} we report the {computational} results for
all benchmark instances of the test sets CP and SV.
The first four columns give details about the instance as the number of vertices, the edge density,
the number of edges and an upper bound on the QMST.
The next three columns report the valid lower bound~\eqref{eq:facialred-minimization-oneline} obtained after the first outer loop of our PRSM algorithm, the relative gap to the upper bound  $100(UB-\text{LB}_{DNN})/UB$, and the wall clock time in seconds needed to compute that bound.
The last six columns outline the {computational} results of our algorithm to compute~\eqref{eq:facialred-minimization-oneline2}.
 In columns 8 to 10, we provide the valid lower bound returned by our algorithm, the relative gap, and the wall clock time needed to compute the lower bound.
The next two columns list the total number of iterations and the total number of cuts added.
 In the last column, we report the relative gap closed by adding the RLT-type cuts to the DNN relaxation~\eqref{eq:facialred-minimization-oneline}. 
 This performance measurement is computed as
$100 (\text{LB}_{DNN+CUTS} - \text{LB}_{DNN})/(\text{UB} - \text{LB}_{DNN}),$
where LB\textsubscript{DNN} refers to the lower bound~\eqref{eq:facialred-minimization-oneline} reported in column~5 and LB\textsubscript{DNN+CUTS} is the lower bound~\eqref{eq:facialred-minimization-oneline2} reported in column~8 in each table.
This metric gives information on how much the gap to the upper bound was improved.
\revision{For a quick overview, we print the average gaps at the end of each table.}

\Cref{tab:CP1,tab:CP2,tab:CP3,tab:CP4} show that especially for CP instances with $n \geq 30$ vertices and edge density~100\% there were only a few violated cuts found.
Hence, the relative improvement of the DNN relaxation by adding those cuts was only marginal. One can further observe that the improvement of the relative gap and the relative gap closed, is better for smaller instances.
For larger instances, adding  cuts such as the RLT-type of the cut-set constraints
for subsets~$S$ of size~2 and larger, might further improve the DNN bounds.

\Cref{tab:SV} presents the results of our algorithm for the benchmark set~SV introduced by~\citet{Sotirov2024}. To the best of our knowledge, there are no results on LAGN, LAGP, and RLT2 for this benchmark set.
The by far best lower bound up to date for the SV instance set was VS2. 
Our DNN relaxation bound without cuts outperforms VS2 for all instances, with the number
of edges $m \geq 45$, except for the instance with $n = 12$ and $d = 67\%$.
Both our relaxations yield a relative gap of less than~1\%. The relative gap of VS2 ranges between~0 and~16.4\%.
The maximum runtime to compute the DNN bound for these instances is less than~5~seconds,
whereas computation time for bound VS2 of~$n = 30$ and~$d = 100\%$ was reported to be~45 minutes.
Computing the DNN bound with cuts is faster than the reported time to compute VS2 for all
instances with more than~80 edges.

\Cref{tab:OPsym,tab:OPesym,tab:OPvsym} read similarly to the tables for the CP and SV benchmark sets but the results are averaged over all instances of the same size.
Again, to the best of our knowledge, we are not aware of any detailed and complete results for LAGN and LAGP on the OP benchmark set.

\Cref{tab:OPsym} reports the results obtained for the benchmark set OPsym.
The lower bound~\eqref{eq:facialred-minimization-oneline2} with cuts 
outperforms VS2 for~$n \geq 10$, and RLT2 for~$n \geq 8$ with the
exception of~$n = 18$, where the average relative gap for RLT2 is reported to be~33\%
and is~33.41\% for the DNN bound with cuts. For~$n = 50$, no bounds were reported.
One can observe that the absolute improvement by adding RLT cuts to~\eqref{eq:facialred-minimization-oneline}
for $n \geq 9$ is approximately~20.

\Cref{tab:OPesym} shows that for the benchmark set OPesym adding the RLT-type
cuts to~\eqref{eq:facialred-minimization-oneline} yields a substantial improvement of the relative gap.
The DNN lower bound with cuts yields better bounds compared to VS2 but is
clearly dominated by RLT1, giving an average relative gap between~$0.2\%$ and~$1.7\%$
for instances with~$n \leq 30$.

\citet{Sotirov2024} report that the relative gap of the VS1 lower bound
is less than or equal~0.2\% for all instances of the class OPvsym.
Although, on average, not many violated cuts to be added were found, the averaged relative bound closed is above~49\% for all instances except that with~$n \in \{6,7\}$, where on average only~0.5~violated cuts were found. Considering the instances with~$n \geq 11$, the average relative bound closed is even above~80\%.

The time limit of~3~hours was reached by all instances from OPesym and OPvsym of size~$n = 50$ and almost all of those instances of size~$n = 30$.
The higher computational costs for those two classes of benchmark instances can be explained,
among other things, by the high number of clusters~$N_{\max}$, cf.~\Cref{subsect:cuttingplane}.
The number of clusters has a direct effect on the computation time of Dykstra's algorithm, which
accounts for a substantial part of the overall computation time.
The average number of clusters needed for the OPvsym and OPesym instances are~6.43 and~6.38,
whereas the average over all other benchmark instances is~3.26. 
Note that for those two classes of instances, added RTL-type constraints significantly improve lower bounds.
Additionally, as for the CP3 instances, one can observe the higher number of iterations until convergence of the algorithm compared to other classes in our benchmark sets.

\section{Conclusion}

This paper provides {a mixed-integer semidefinite programming formulation} for the quadratic minimum spanning tree problem.
{This formulation} includes only one connectivity constraint, which is a linear matrix inequality based on the algebraic connectivity of trees.
By exploiting {the MISDP formulation}, we derive a DNN relaxation for the QMSTP.
We also derive the cut-set and  RLT-type constraints as Chv\'atal-Gomory cuts of the MISDP by applying a CG procedure for mixed integer conic programming.
The RLT-type constraints are  added to the DNN relaxation, resulting in a strengthened DNN relaxation.
An iterative cutting plane Peaceman-Rachford splitting method is designed to compute
the DNN relaxation with the RLT-type  constraints of the QMSTP efficiently.

The computational experiments on the benchmark instances from the literature
demonstrate that our bounds significantly outperform existing bounds in quality
and computation time.
While other approaches struggled to compute bounds for larger instances, we compute strong bounds in short time. 

Given these results, incorporating our new bounds in a branch-and-bound algorithm would be the obvious next step for further research. Another topic for future research would be to incorporate additional RLT-type cut-set constraints to further strengthen the DNN relaxation.

\vspace{0.5cm}
{

\begingroup
\setlength{\tabcolsep}{7pt} 
\renewcommand{\arraystretch}{1.1}
\begin{table}[H]
\centering
\footnotesize
\begin{tabular}{lrrrr|rrrr}
        \toprule
       & \multicolumn{4}{c}{This study} & \multicolumn{4}{c}{\citet{GUIMARAES202046}}\\
       \cmidrule(r){2-5} \cmidrule(l){6-9}
       & \multicolumn{2}{c}{DNN} & \multicolumn{2}{c}{DNN + CUTS}& \multicolumn{2}{c}{LAGN } & \multicolumn{2}{c}{LAGP}\\
       \cmidrule(r){2-3} \cmidrule(lr){4-5} \cmidrule(lr){6-7} \cmidrule(l){8-9}
       & \multicolumn{1}{l}{gap (\%)}      & \multicolumn{1}{l}{time (s)}     & \multicolumn{1}{l}{gap (\%)}        & \multicolumn{1}{l}{time (s)}   & \multicolumn{1}{l}{gap (\%)}       & \multicolumn{1}{l}{time (s)}   & \multicolumn{1}{l}{gap (\%)}  & \multicolumn{1}{l}{time (s)}     \\[0.2em] \hline
\begin{filecontents*}{result-CP-navg.csv}
size;text;dnngap;dnngapp;dnntime;lbgap;lbgapp;lbtime;lagngap;lagntime;lagpgap;lagptime
10; $n  =  10$ ;0.05;4.64;0.08;0.02;2.03;9.31;3.78;0.41;3.43;1.03
15; $n  =  15$ ;0.06;5.54;0.41;0.04;4.24;18.77;5.25;4.81;8.54;6.59
20; $n  =  20$ ;0.06;5.91;1.02;0.05;5.32;27.3;6.43;30.64;13.32;32.62
25; $n  =  25$ ;0.08;7.52;2.32;0.07;6.83;35.8;8.72;122.87;17.56;239.23
30; $n  =  30$ ;0.09;8.66;5.85;0.08;8.37;54.1;11.86;358.2;21.91;892.74
35; $n  =  35$ ;0.1;9.87;7.62;0.1;9.68;58.7;15.96;897.69;24.03;3582.2
40; $n  =  40$ ;0.11;11.33;14.34;0.11;11.19;68.1;23.53;1597.73;28.72;6050.58
45; $n  =  45$ ;0.12;11.88;27.13;0.12;11.78;84.4;27.34;3195.97;29.51;16297.23
50; $n  =  50$ ;0.13;13.08;45.58;0.13;13.03;103.22;31.45;6030;31.99;31953.35
\end{filecontents*}
\csvreader[head to column names, late after line = \\, /csv/separator=semicolon]{result-CP-navg.csv}{}
        {\text & \dnngapp & \dnntime & \lbgapp & \lbtime & \lagngap & \lagntime & \lagpgap &\lagptime}
\hline
\begin{filecontents*}{result-CP-davg.csv}
d;text;dnngap;dnngapp;dnntime;lbgap;lbgapp;lbtime;lagngap;lagntime;lagpgap;lagptime
33; $d  =  33\%$ ;0.05;4.93;1.97;0.04;3.97;23.31;4.84;145.89;9.8;191.42
67; $d  =  67\%$ ;0.09;8.99;7.02;0.08;8.17;44.53;14.28;1124.67;20.96;3968.84
100; $d  =  100\%$ ;0.12;12.22;25.79;0.12;12.01;85.39;25.65;2808.88;28.92;15524.93
\end{filecontents*}
\csvreader[head to column names, late after line = \\, /csv/separator=semicolon]{result-CP-davg.csv}{}
        {\text & \dnngapp & \dnntime & \lbgapp & \lbtime & \lagngap & \lagntime & \lagpgap &\lagptime}
\hline
\begin{filecontents*}{result-CP-classavg.csv}
class;text;dnngap;dnngapp;dnntime;lbgap;lbgapp;lbtime;lagngap;lagntime;lagpgap;lagptime
1; class  =  1 ;0.09;9.08;8.14;0.09;8.51;21.87;17.93;551.84;20.66;5270.2
2; class  =  2 ;0.11;10.96;15.69;0.1;10.23;45.99;17.27;1629.16;23.73;8583.7
3; class  =  3 ;0.04;4.35;13.49;0.04;3.63;106.83;7.63;1625.35;12.14;4462.49
4; class  =  4 ;0.1;10.47;9.05;0.1;9.84;29.61;16.87;1632.91;23.04;7930.53
\end{filecontents*}
\csvreader[head to column names, late after line = \\, /csv/separator=semicolon]{result-CP-classavg.csv}{}
        {\text & \dnngapp & \dnntime & \lbgapp & \lbtime & \lagngap & \lagntime & \lagpgap &\lagptime}
\bottomrule
\end{tabular}
\caption{Comparison to averaged results on lower bounds for CP instances. \revision{Note that the experiments were carried out on different machines, see Section~\ref{sec:numericalresults} for the details.}}\label{tab:comparison-bbn}
\end{table}
\endgroup
}

{ \footnotesize
\begingroup
\setlength{\tabcolsep}{4pt} 
\renewcommand{\arraystretch}{1.1} 
    \begin{longtable}{lrrr|rrr|rrrrrr}
    \toprule
    \multicolumn{4}{c}{Instance} & \multicolumn{3}{c}{DNN} & \multicolumn{6}{c}{DNN + CUTS}\\
    \cmidrule(r){1-4} \cmidrule(lr){5-7} \cmidrule(l){8-13}
         \multicolumn{1}{l}{$n$}& \multicolumn{1}{l}{$d$ (\%)} & \multicolumn{1}{l}{$m$} & \multicolumn{1}{l}{UB} & \multicolumn{1}{l}{LB}  & \multicolumn{1}{l}{gap (\%)} & \multicolumn{1}{l}{time (s)} & \multicolumn{1}{l}{LB}  & \multicolumn{1}{l}{gap (\%)} & \multicolumn{1}{l}{time (s)} & \multicolumn{1}{l}{iterations} & \multicolumn{1}{l}{cuts} & \multicolumn{1}{l}{closed (\%)} \\\midrule \endfirsthead
        \multicolumn{1}{l}{$n$}& \multicolumn{1}{l}{$d$ (\%)} & \multicolumn{1}{l}{$m$} & \multicolumn{1}{l}{UB} & \multicolumn{1}{l}{LB}  & \multicolumn{1}{l}{gap (\%)} & \multicolumn{1}{l}{time (s)} & \multicolumn{1}{l}{LB}  & \multicolumn{1}{l}{gap (\%)} & \multicolumn{1}{l}{time (s)} & \multicolumn{1}{l}{iterations} & \multicolumn{1}{l}{cuts} & \multicolumn{1}{l}{closed (\%)}\\\midrule\endhead
        \bottomrule\\
        \caption{Results for CP1.}\label{tab:CP1}\endlastfoot
        \begin{filecontents*}{result-CP1.csv}
instance;d;class;n;m;ub;dnn;dnnrounded;dnngap;dnngapp;dnntime;lb;lbrounded;lbgap;lbgapp;time;ncuts;nclusters;iterations;outerloops;solutionstatus;relgapclosed
qmstp\_CP10\_33\_10\_10;33;1;10;14;350;341.08;342;0.03;2.55;0.02;349.97;350;0.00;0.01;0.74;28;1;502;3;FEW\_VIOLATIONS\_FOUND;99.63
qmstp\_CP10\_67\_10\_10;67;1;10;30;255;242.83;243;0.05;4.77;0.14;251.43;252;0.01;1.40;10.08;48;3;816;3;FEW\_VIOLATIONS\_FOUND;70.69
qmstp\_CP10\_100\_10\_10;100;1;10;45;239;226.33;227;0.05;5.30;0.19;228.96;229;0.04;4.20;13.87;60;3;820;3;FEW\_VIOLATIONS\_FOUND;20.75
qmstp\_CP15\_33\_10\_10;33;1;15;34;745;706.16;707;0.05;5.21;0.13;719.04;720;0.03;3.48;8.73;91;3;810;4;FEW\_VIOLATIONS\_FOUND;33.17
qmstp\_CP15\_67\_10\_10;67;1;15;70;659;619.40;620;0.06;6.01;0.36;628.71;629;0.05;4.60;7.16;140;3;796;4;FEW\_VIOLATIONS\_FOUND;23.53
qmstp\_CP15\_100\_10\_10;100;1;15;105;620;585.24;586;0.06;5.61;0.67;587.03;588;0.05;5.32;4.70;105;3;286;2;FEW\_VIOLATIONS\_FOUND;5.12
qmstp\_CP20\_33\_10\_10;33;1;20;62;1379;1318.35;1319;0.04;4.40;0.18;1328.45;1329;0.04;3.67;4.99;143;3;530;4;FEW\_VIOLATIONS\_FOUND;16.66
qmstp\_CP20\_67\_10\_10;67;1;20;127;1252;1171.25;1172;0.06;6.45;0.99;1179.66;1180;0.06;5.78;10.62;221;4;511;4;FEW\_VIOLATIONS\_FOUND;10.41
qmstp\_CP20\_100\_10\_10;100;1;20;190;1174;1072.25;1073;0.09;8.67;0.92;1074.01;1075;0.09;8.52;7.81;168;5;166;2;FEW\_VIOLATIONS\_FOUND;1.73
qmstp\_CP25\_33\_10\_10;33;1;25;99;2185;2101.49;2102;0.04;3.82;0.75;2134.53;2135;0.02;2.31;12.94;289;3;787;5;FEW\_VIOLATIONS\_FOUND;39.56
qmstp\_CP25\_67\_10\_10;67;1;25;201;2023;1863.59;1864;0.08;7.88;0.81;1866.65;1867;0.08;7.73;8.23;228;4;163;3;FEW\_VIOLATIONS\_FOUND;1.92
qmstp\_CP25\_100\_10\_10;100;1;25;300;1943;1705.87;1706;0.12;12.20;2.44;1706.67;1707;0.12;12.16;10.60;138;3;155;2;FEW\_VIOLATIONS\_FOUND;0.34
qmstp\_CP30\_33\_10\_10;33;1;30;143;3205;3073.14;3074;0.04;4.11;1.58;3089.28;3090;0.04;3.61;11.21;334;3;584;4;FEW\_VIOLATIONS\_FOUND;12.24
qmstp\_CP30\_67\_10\_10;67;1;30;291;2998;2709.48;2710;0.10;9.62;1.94;2712.29;2713;0.10;9.53;10.25;284;3;151;3;SLOW\_IMPROVEMENT;0.97
qmstp\_CP30\_100\_10\_10;100;1;30;435;2874;2475.81;2476;0.14;13.86;4.31;2476.25;2477;0.14;13.84;16.19;102;3;165;2;FEW\_VIOLATIONS\_FOUND;0.11
qmstp\_CP35\_33\_10\_10;33;1;35;196;4474;4235.79;4236;0.05;5.32;1.49;4251.43;4252;0.05;4.97;12.83;410;3;381;4;FEW\_VIOLATIONS\_FOUND;6.56
qmstp\_CP35\_67\_10\_10;67;1;35;398;4147;3715.44;3716;0.10;10.41;2.66;3717.16;3718;0.10;10.37;12.92;253;3;133;2;FEW\_VIOLATIONS\_FOUND;0.40
qmstp\_CP35\_100\_10\_10;100;1;35;595;4000;3399.70;3400;0.15;15.01;11.80;3399.99;3400;0.15;15.00;30.73;90;2;177;2;FEW\_VIOLATIONS\_FOUND;0.05
qmstp\_CP40\_33\_10\_10;33;1;40;257;5945;5573.59;5574;0.06;6.25;1.24;5587.07;5588;0.06;6.02;12.56;456;3;188;4;FEW\_VIOLATIONS\_FOUND;3.63
qmstp\_CP40\_67\_10\_10;67;1;40;522;5567;4878.53;4879;0.12;12.37;6.82;4879.76;4880;0.12;12.34;24.27;227;3;142;2;FEW\_VIOLATIONS\_FOUND;0.18
qmstp\_CP40\_100\_10\_10;100;1;40;780;5368;4457.25;4458;0.17;16.97;21.13;4457.41;4458;0.17;16.96;51.71;78;2;189;2;FEW\_VIOLATIONS\_FOUND;0.02
qmstp\_CP45\_33\_10\_10;33;1;45;326;7521;7065.52;7066;0.06;6.06;2.97;7080.87;7081;0.06;5.85;15.74;584;3;255;3;FEW\_VIOLATIONS\_FOUND;3.37
qmstp\_CP45\_67\_10\_10;67;1;45;663;7161;6210.43;6211;0.13;13.27;11.60;6210.88;6211;0.13;13.27;30.85;170;2;142;2;FEW\_VIOLATIONS\_FOUND;0.05
qmstp\_CP45\_100\_10\_10;100;1;45;990;6944;5668.59;5669;0.18;18.37;43.86;5668.64;5669;0.18;18.37;77.93;33;1;193;2;FEW\_VIOLATIONS\_FOUND;0.00
qmstp\_CP50\_33\_10\_10;33;1;50;404;9393;8737.94;8738;0.07;6.97;2.32;8749.60;8750;0.07;6.85;14.86;564;2;149;3;FEW\_VIOLATIONS\_FOUND;1.78
qmstp\_CP50\_67\_10\_10;67;1;50;820;8958;7668.14;7669;0.14;14.40;19.19;7668.62;7669;0.14;14.39;48.21;152;2;156;2;FEW\_VIOLATIONS\_FOUND;0.04
qmstp\_CP50\_100\_10\_10;100;1;50;1225;8713;7029.42;7030;0.19;19.32;79.40;7029.45;7030;0.19;19.32;119.82;32;1;206;2;FEW\_VIOLATIONS\_FOUND;0.00
\end{filecontents*}
\csvreader[head to column names, late after line = \\, /csv/separator=semicolon]{result-CP1.csv}{n=\size,m=\edges,d=\density}
        {\size & \density & \edges & \ub & \dnn & \dnngapp & \dnntime & \lb  & \lbgapp & \time & \iterations & \ncuts & \relgapclosed}
    \cmidrule(r){1-4} \cmidrule(lr){5-7} \cmidrule(l){8-13}        
    \multicolumn{4}{l|}{\revision{Average}} & & \revision{9.08} & & & \revision{8.51} & & & & \revision{13.07}
    \end{longtable}
\endgroup
}
    \vspace{-0.3cm}
    { \footnotesize
\begingroup
\setlength{\tabcolsep}{4pt} 
\renewcommand{\arraystretch}{1.1} 
        \begin{longtable}{lrrr|rrr|rrrrrr}
    \toprule
    \multicolumn{4}{c}{Instance} & \multicolumn{3}{c}{DNN} & \multicolumn{6}{c}{DNN + CUTS}\\
    \cmidrule(r){1-4} \cmidrule(lr){5-7} \cmidrule(l){8-13}
         \multicolumn{1}{l}{$n$}& \multicolumn{1}{l}{$d$ (\%)} & \multicolumn{1}{l}{$m$} & \multicolumn{1}{l}{UB} & \multicolumn{1}{l}{LB}  & \multicolumn{1}{l}{gap (\%)} & \multicolumn{1}{l}{time (s)} & \multicolumn{1}{l}{LB}  & \multicolumn{1}{l}{gap (\%)} & \multicolumn{1}{l}{time (s)} & \multicolumn{1}{l}{iterations} & \multicolumn{1}{l}{cuts} & \multicolumn{1}{l}{closed (\%)} \\\midrule \endfirsthead
        \multicolumn{1}{l}{$n$}& \multicolumn{1}{l}{$d$ (\%)} & \multicolumn{1}{l}{$m$} & \multicolumn{1}{l}{UB} & \multicolumn{1}{l}{LB}  & \multicolumn{1}{l}{gap (\%)} & \multicolumn{1}{l}{time (s)} & \multicolumn{1}{l}{LB}  & \multicolumn{1}{l}{gap (\%)} & \multicolumn{1}{l}{time (s)} & \multicolumn{1}{l}{iterations} & \multicolumn{1}{l}{cuts} & \multicolumn{1}{l}{closed (\%)}\\\midrule\endhead
        \bottomrule\\
        \caption{Results for CP2.}\label{tab:CP2}\endlastfoot
        \begin{filecontents*}{result-CP2.csv}
instance;d;class;n;m;ub;dnn;dnnrounded;dnngap;dnngapp;dnntime;lb;lbrounded;lbgap;lbgapp;time;ncuts;nclusters;iterations;outerloops;solutionstatus;relgapclosed
qmstp\_CP10\_33\_10\_100;33;2;10;14;3122;2958.98;2959;0.05;5.22;0.01;3115.06;3116;0.00;0.22;2.88;28;1;579;3;FEW\_VIOLATIONS\_FOUND;95.75
qmstp\_CP10\_67\_10\_100;67;2;10;30;2042;1929.22;1930;0.06;5.52;0.08;2027.39;2028;0.01;0.72;25.75;55;2;923;3;FEW\_VIOLATIONS\_FOUND;87.05
qmstp\_CP10\_100\_10\_100;100;2;10;45;1815;1681.60;1682;0.07;7.35;0.14;1697.01;1698;0.07;6.50;24.15;50;3;949;3;FEW\_VIOLATIONS\_FOUND;11.55
qmstp\_CP15\_33\_10\_100;33;2;15;34;6539;6108.39;6109;0.07;6.59;0.05;6236.79;6237;0.05;4.62;9.54;83;2;666;4;FEW\_VIOLATIONS\_FOUND;29.82
qmstp\_CP15\_67\_10\_100;67;2;15;70;5573;5182.26;5183;0.07;7.01;0.25;5273.86;5274;0.05;5.37;10.19;120;3;651;3;FEW\_VIOLATIONS\_FOUND;23.44
qmstp\_CP15\_100\_10\_100;100;2;15;105;5184;4802.29;4803;0.07;7.36;0.51;4819.98;4820;0.07;7.02;19.41;126;4;461;3;FEW\_VIOLATIONS\_FOUND;4.63
qmstp\_CP20\_33\_10\_100;33;2;20;62;12425;11770.26;11771;0.05;5.27;0.12;11860.87;11861;0.05;4.54;9.11;140;3;549;4;FEW\_VIOLATIONS\_FOUND;13.84
qmstp\_CP20\_67\_10\_100;67;2;20;127;10893;10251.15;10252;0.06;5.89;0.66;10316.10;10317;0.05;5.30;13.62;182;4;389;3;FEW\_VIOLATIONS\_FOUND;10.12
qmstp\_CP20\_100\_10\_100;100;2;20;190;10215;9138.20;9139;0.11;10.54;1.53;9156.76;9157;0.10;10.36;14.36;153;3;346;2;FEW\_VIOLATIONS\_FOUND;1.72
qmstp\_CP25\_33\_10\_100;33;2;25;99;19976;19156.48;19157;0.04;4.10;0.56;19499.99;19500;0.02;2.38;35.03;282;3;1054;5;FEW\_VIOLATIONS\_FOUND;41.92
qmstp\_CP25\_67\_10\_100;67;2;25;201;18251;16563.59;16564;0.09;9.25;1.36;16595.43;16596;0.09;9.07;19.68;221;4;323;3;FEW\_VIOLATIONS\_FOUND;1.89
qmstp\_CP25\_100\_10\_100;100;2;25;300;17411;14772.45;14773;0.15;15.15;3.40;14782.12;14783;0.15;15.10;28.02;140;3;346;2;FEW\_VIOLATIONS\_FOUND;0.37
qmstp\_CP30\_33\_10\_100;33;2;30;143;29731;28301.88;28302;0.05;4.81;0.78;28469.60;28470;0.04;4.24;22.82;326;3;496;4;FEW\_VIOLATIONS\_FOUND;11.74
qmstp\_CP30\_67\_10\_100;67;2;30;291;27581;24303.73;24304;0.12;11.88;2.86;24330.85;24331;0.12;11.78;22.24;259;3;318;3;SLOW\_IMPROVEMENT;0.83
qmstp\_CP30\_100\_10\_100;100;2;30;435;26146;21699.39;21700;0.17;17.01;9.23;21703.81;21704;0.17;16.99;35.44;90;2;361;2;FEW\_VIOLATIONS\_FOUND;0.10
qmstp\_CP35\_33\_10\_100;33;2;35;196;42305;39405.07;39406;0.07;6.85;0.98;39551.95;39552;0.07;6.51;19.02;400;3;337;4;SLOW\_IMPROVEMENT;5.06
qmstp\_CP35\_67\_10\_100;67;2;35;398;38490;33520.06;33521;0.13;12.91;5.37;33537.58;33538;0.13;12.87;35.03;258;3;318;3;SLOW\_IMPROVEMENT;0.35
qmstp\_CP35\_100\_10\_100;100;2;35;595;36723;30027.26;30028;0.18;18.23;20.46;30030.25;30031;0.18;18.22;63.99;83;2;388;2;FEW\_VIOLATIONS\_FOUND;0.04
qmstp\_CP40\_33\_10\_100;33;2;40;257;56237;51995.88;51996;0.08;7.54;1.83;52136.64;52137;0.07;7.29;19.07;437;3;284;4;SLOW\_IMPROVEMENT;3.32
qmstp\_CP40\_67\_10\_100;67;2;40;522;51851;44247.99;44248;0.15;14.66;11.95;44260.68;44261;0.15;14.64;53.48;238;3;313;2;FEW\_VIOLATIONS\_FOUND;0.17
qmstp\_CP40\_100\_10\_100;100;2;40;780;49817;39565.62;39566;0.21;20.58;40.64;39567.19;39568;0.21;20.57;113.05;63;2;413;2;FEW\_VIOLATIONS\_FOUND;0.02
qmstp\_CP45\_33\_10\_100;33;2;45;326;70603;65914.02;65915;0.07;6.64;2.38;66041.56;66042;0.06;6.46;32.93;589;3;310;4;SLOW\_IMPROVEMENT;2.72
qmstp\_CP45\_67\_10\_100;67;2;45;663;66889;56511.11;56512;0.16;15.52;22.95;56515.55;56516;0.16;15.51;63.23;151;2;313;2;FEW\_VIOLATIONS\_FOUND;0.04
qmstp\_CP45\_100\_10\_100;100;2;45;990;64840;50526.76;50527;0.22;22.07;88.58;50527.18;50528;0.22;22.07;152.35;42;1;419;2;FEW\_VIOLATIONS\_FOUND;0.00
qmstp\_CP50\_33\_10\_100;33;2;50;404;88942;81835.33;81836;0.08;7.99;3.48;81933.27;81934;0.08;7.88;27.64;557;2;262;3;FEW\_VIOLATIONS\_FOUND;1.38
qmstp\_CP50\_67\_10\_100;67;2;50;820;84020;69935.17;69936;0.17;16.76;37.26;69939.94;69940;0.17;16.76;101.24;145;2;342;2;FEW\_VIOLATIONS\_FOUND;0.03
qmstp\_CP50\_100\_10\_100;100;2;50;1225;81858;62875.12;62876;0.23;23.19;166.17;62875.44;62876;0.23;23.19;268.61;30;2;453;2;FEW\_VIOLATIONS\_FOUND;0.00
\end{filecontents*}
\csvreader[head to column names, late after line = \\, /csv/separator=semicolon]{result-CP2.csv}{n=\size,m=\edges,d=\density}
        {\size & \density & \edges & \ub & \dnn & \dnngapp & \dnntime & \lb  & \lbgapp & \time & \iterations & \ncuts & \relgapclosed}
    \cmidrule(r){1-4} \cmidrule(lr){5-7} \cmidrule(l){8-13}        
    \multicolumn{4}{l|}{\revision{Average}} & & \revision{10.96} & & & \revision{10.23} & & & & \revision{12.89}

    \end{longtable}
\endgroup
}

{\footnotesize
\begingroup
\setlength{\tabcolsep}{4pt} 
\renewcommand{\arraystretch}{1.1}
        \begin{longtable}{lrrr|rrr|rrrrrr}
    \toprule
    \multicolumn{4}{c}{Instance} & \multicolumn{3}{c}{DNN} & \multicolumn{6}{c}{DNN + CUTS}\\
    \cmidrule(r){1-4} \cmidrule(lr){5-7} \cmidrule(l){8-13}
         \multicolumn{1}{l}{$n$}& \multicolumn{1}{l}{$d$ (\%)} & \multicolumn{1}{l}{$m$} & \multicolumn{1}{l}{UB} & \multicolumn{1}{l}{LB}  & \multicolumn{1}{l}{gap (\%)} & \multicolumn{1}{l}{time (s)} & \multicolumn{1}{l}{LB}  & \multicolumn{1}{l}{gap (\%)} & \multicolumn{1}{l}{time (s)} & \multicolumn{1}{l}{iterations} & \multicolumn{1}{l}{cuts} & \multicolumn{1}{l}{closed (\%)} \\\midrule \endfirsthead
        \multicolumn{1}{l}{$n$}& \multicolumn{1}{l}{$d$ (\%)} & \multicolumn{1}{l}{$m$} & \multicolumn{1}{l}{UB} & \multicolumn{1}{l}{LB}  & \multicolumn{1}{l}{gap (\%)} & \multicolumn{1}{l}{time (s)} & \multicolumn{1}{l}{LB}  & \multicolumn{1}{l}{gap (\%)} & \multicolumn{1}{l}{time (s)} & \multicolumn{1}{l}{iterations} & \multicolumn{1}{l}{cuts} & \multicolumn{1}{l}{closed (\%)}\\\midrule\endhead
        \bottomrule\\
        \caption{Results for CP3.}\label{tab:CP3}\endlastfoot
        \begin{filecontents*}{result-CP3.csv}
instance;d;class;n;m;ub;dnn;dnnrounded;dnngap;dnngapp;dnntime;lb;lbrounded;lbgap;lbgapp;time;ncuts;nclusters;iterations;outerloops;solutionstatus;relgapclosed
qmstp\_CP10\_33\_100\_10;33;3;10;14;646;625.91;626;0.03;3.11;0.01;625.91;626;0.03;3.11;0.01;0;0;130;1;FEW\_VIOLATIONS\_FOUND;0.00
qmstp\_CP10\_67\_100\_10;67;3;10;30;488;460.23;461;0.06;5.69;0.04;487.03;488;0.00;0.20;3.68;26;2;538;2;FEW\_VIOLATIONS\_FOUND;96.52
qmstp\_CP10\_100\_100\_10;100;3;10;45;426;422.02;423;0.01;0.93;0.09;424.27;425;0.00;0.41;6.67;18;3;556;2;FEW\_VIOLATIONS\_FOUND;56.51
qmstp\_CP15\_33\_100\_10;33;3;15;34;1236;1203.32;1204;0.03;2.64;0.15;1225.21;1226;0.01;0.87;11.52;79;2;1311;4;FEW\_VIOLATIONS\_FOUND;66.99
qmstp\_CP15\_67\_100\_10;67;3;15;70;966;931.32;932;0.04;3.59;0.25;946.11;947;0.02;2.06;39.28;89;4;902;3;FEW\_VIOLATIONS\_FOUND;42.66
qmstp\_CP15\_100\_100\_10;100;3;15;105;975;953.87;954;0.02;2.17;1.59;961.97;962;0.01;1.34;89.76;145;6;1860;3;FEW\_VIOLATIONS\_FOUND;38.32
qmstp\_CP20\_33\_100\_10;33;3;20;62;1972;1906.69;1907;0.03;3.31;0.29;1924.45;1925;0.02;2.41;21.52;178;3;1088;4;FEW\_VIOLATIONS\_FOUND;27.19
qmstp\_CP20\_67\_100\_10;67;3;20;127;1792;1742.91;1743;0.03;2.74;1.52;1762.12;1763;0.02;1.67;101.72;311;5;1303;4;FEW\_VIOLATIONS\_FOUND;39.13
qmstp\_CP20\_100\_100\_10;100;3;20;190;1544;1514.39;1515;0.02;1.92;4.25;1520.24;1521;0.02;1.54;122.27;248;8;1212;3;FEW\_VIOLATIONS\_FOUND;19.78
qmstp\_CP25\_33\_100\_10;33;3;25;99;2976;2861.20;2862;0.04;3.86;1.01;2904.43;2905;0.02;2.40;25.98;281;3;1353;5;FEW\_VIOLATIONS\_FOUND;37.66
qmstp\_CP25\_67\_100\_10;67;3;25;201;2546;2468.53;2469;0.03;3.04;5.70;2487.99;2488;0.02;2.28;100.34;487;5;1415;5;FEW\_VIOLATIONS\_FOUND;25.13
qmstp\_CP25\_100\_100\_10;100;3;25;300;2471;2371.04;2372;0.04;4.05;7.83;2382.68;2383;0.04;3.57;152.13;438;6;811;3;FEW\_VIOLATIONS\_FOUND;11.65
qmstp\_CP30\_33\_100\_10;33;3;30;143;4070;4001.92;4002;0.02;1.67;2.55;4027.98;4028;0.01;1.03;39.17;460;4;1278;5;FEW\_VIOLATIONS\_FOUND;38.28
qmstp\_CP30\_67\_100\_10;67;3;30;291;3649;3522.91;3523;0.03;3.46;6.93;3544.28;3545;0.03;2.87;131.82;635;5;921;4;FEW\_VIOLATIONS\_FOUND;16.95
qmstp\_CP30\_100\_100\_10;100;3;30;435;3483;3293.54;3294;0.05;5.44;12.41;3305.31;3306;0.05;5.10;189.36;665;7;684;4;FEW\_VIOLATIONS\_FOUND;6.21
qmstp\_CP35\_33\_100\_10;33;3;35;196;5423;5261.95;5262;0.03;2.97;3.76;5302.59;5303;0.02;2.22;59.07;622;3;1361;5;FEW\_VIOLATIONS\_FOUND;25.24
qmstp\_CP35\_67\_100\_10;67;3;35;398;4981;4757.51;4758;0.04;4.49;7.89;4767.61;4768;0.04;4.28;113.69;627;4;593;4;FEW\_VIOLATIONS\_FOUND;4.52
qmstp\_CP35\_100\_100\_10;100;3;35;595;4770;4451.46;4452;0.07;6.68;21.04;4454.47;4455;0.07;6.61;298.89;593;5;503;3;FEW\_VIOLATIONS\_FOUND;0.94
qmstp\_CP40\_33\_100\_10;33;3;40;257;6925;6708.96;6709;0.03;3.12;5.77;6753.53;6754;0.02;2.48;73.30;849;3;1099;5;FEW\_VIOLATIONS\_FOUND;20.63
qmstp\_CP40\_67\_100\_10;67;3;40;522;6456;6110.47;6111;0.05;5.35;12.16;6122.99;6123;0.05;5.16;136.20;776;4;482;4;SLOW\_IMPROVEMENT;3.62
qmstp\_CP40\_100\_100\_10;100;3;40;780;6208;5716.76;5717;0.08;7.91;37.61;5718.10;5719;0.08;7.89;183.36;422;4;446;3;SLOW\_IMPROVEMENT;0.27
qmstp\_CP45\_33\_100\_10;33;3;45;326;8720;8484.37;8485;0.03;2.70;9.77;8517.33;8518;0.02;2.32;81.91;1035;4;1051;5;FEW\_VIOLATIONS\_FOUND;13.99
qmstp\_CP45\_67\_100\_10;67;3;45;663;8225;7692.48;7693;0.06;6.47;18.71;7700.56;7701;0.06;6.38;147.23;884;4;373;3;FEW\_VIOLATIONS\_FOUND;1.52
qmstp\_CP45\_100\_100\_10;100;3;45;990;7827;7160.41;7161;0.09;8.52;67.09;7161.31;7162;0.09;8.50;288.73;382;3;384;3;SLOW\_IMPROVEMENT;0.14
qmstp\_CP50\_33\_100\_10;33;3;50;404;10717;10341.87;10342;0.04;3.50;12.76;10367.75;10368;0.03;3.26;89.62;1101;4;851;5;FEW\_VIOLATIONS\_FOUND;6.90
qmstp\_CP50\_67\_100\_10;67;3;50;820;10100;9366.40;9367;0.07;7.26;28.69;9368.84;9369;0.07;7.24;157.12;660;3;286;3;SLOW\_IMPROVEMENT;0.33
qmstp\_CP50\_100\_100\_10;100;3;50;1225;9836;8767.89;8768;0.11;10.86;94.33;8768.14;8769;0.11;10.86;220.08;166;3;292;2;FEW\_VIOLATIONS\_FOUND;0.02
\end{filecontents*}
\csvreader[head to column names, late after line = \\, /csv/separator=semicolon]{result-CP3.csv}{n=\size,m=\edges,d=\density}
        {\size & \density & \edges & \ub & \dnn & \dnngapp & \dnntime & \lb  & \lbgapp & \time & \iterations & \ncuts & \relgapclosed}
    \cmidrule(r){1-4} \cmidrule(lr){5-7} \cmidrule(l){8-13}        
    \multicolumn{4}{l|}{\revision{Average}} & & \revision{4.35} & & & \revision{3.63} & & & & \revision{22.26} 
    \end{longtable}
\endgroup
}

    {\footnotesize 
\begingroup
\setlength{\tabcolsep}{4pt} 
\renewcommand{\arraystretch}{1.1} 
        \begin{longtable}{lrrr|rrr|rrrrrr}
    \toprule
    \multicolumn{4}{c}{Instance} & \multicolumn{3}{c}{DNN} & \multicolumn{6}{c}{DNN + CUTS}\\
    \cmidrule(r){1-4} \cmidrule(lr){5-7} \cmidrule(l){8-13}
         \multicolumn{1}{l}{$n$}& \multicolumn{1}{l}{$d$ (\%)} & \multicolumn{1}{l}{$m$} & \multicolumn{1}{l}{UB} & \multicolumn{1}{l}{LB}  & \multicolumn{1}{l}{gap (\%)} & \multicolumn{1}{l}{time (s)} & \multicolumn{1}{l}{LB}  & \multicolumn{1}{l}{gap (\%)} & \multicolumn{1}{l}{time (s)} & \multicolumn{1}{l}{iterations} & \multicolumn{1}{l}{cuts} & \multicolumn{1}{l}{closed (\%)} \\\midrule \endfirsthead
        \multicolumn{1}{l}{$n$}& \multicolumn{1}{l}{$d$ (\%)} & \multicolumn{1}{l}{$m$} & \multicolumn{1}{l}{UB} & \multicolumn{1}{l}{LB}  & \multicolumn{1}{l}{gap (\%)} & \multicolumn{1}{l}{time (s)} & \multicolumn{1}{l}{LB}  & \multicolumn{1}{l}{gap (\%)} & \multicolumn{1}{l}{time (s)} & \multicolumn{1}{l}{iterations} & \multicolumn{1}{l}{cuts} & \multicolumn{1}{l}{closed (\%)}\\\midrule\endhead
        \bottomrule\\
        \caption{Results for CP4.}\label{tab:CP4}\endlastfoot
        \begin{filecontents*}{result-CP4.csv}
instance;d;class;n;m;ub;dnn;dnnrounded;dnngap;dnngapp;dnntime;lb;lbrounded;lbgap;lbgapp;time;ncuts;nclusters;iterations;outerloops;solutionstatus;relgapclosed
qmstp\_CP10\_33\_100\_100;33;4;10;14;3486;3391.98;3392;0.03;2.7;0.02;3485.78;3486;0;0.01;2.19;28;1;562;3;FEW\_VIOLATIONS\_FOUND;99.77
qmstp\_CP10\_67\_100\_100;67;4;10;30;2404;2261.59;2262;0.06;5.92;0.09;2348.16;2349;0.02;2.32;10.89;50;3;676;3;FEW\_VIOLATIONS\_FOUND;60.79
qmstp\_CP10\_100\_100\_100;100;4;10;45;2197;2051.05;2052;0.07;6.64;0.16;2082.04;2083;0.05;5.23;10.77;56;3;713;3;FEW\_VIOLATIONS\_FOUND;21.23
qmstp\_CP15\_33\_100\_100;33;4;15;34;7245;6790.59;6791;0.06;6.27;0.16;6958.02;6959;0.04;3.96;10.37;81;3;816;4;FEW\_VIOLATIONS\_FOUND;36.84
qmstp\_CP15\_67\_100\_100;67;4;15;70;6188;5769.57;5770;0.07;6.76;0.29;5862.17;5863;0.05;5.27;9.26;153;3;595;4;FEW\_VIOLATIONS\_FOUND;22.13
qmstp\_CP15\_100\_100\_100;100;4;15;105;5879;5449.98;5450;0.07;7.3;0.47;5470.56;5471;0.07;6.95;5.27;118;4;310;3;FEW\_VIOLATIONS\_FOUND;4.8
qmstp\_CP20\_33\_100\_100;33;4;20;62;13288;12666.64;12667;0.05;4.68;0.18;12766.55;12767;0.04;3.92;4.47;138;3;523;4;FEW\_VIOLATIONS\_FOUND;16.08
qmstp\_CP20\_67\_100\_100;67;4;20;127;11893;11096.91;11097;0.07;6.69;0.76;11178.8;11179;0.06;6.01;7.61;187;4;387;3;FEW\_VIOLATIONS\_FOUND;10.29
qmstp\_CP20\_100\_100\_100;100;4;20;190;11101;9951.75;9952;0.1;10.35;0.88;9971.71;9972;0.1;10.17;9.54;171;4;204;3;SLOW\_IMPROVEMENT;1.74
qmstp\_CP25\_33\_100\_100;33;4;25;99;21176;20227.51;20228;0.04;4.48;0.77;20589.6;20590;0.03;2.77;14.67;288;3;994;5;FEW\_VIOLATIONS\_FOUND;38.18
qmstp\_CP25\_67\_100\_100;67;4;25;201;19207;17574.38;17575;0.09;8.5;0.89;17610.67;17611;0.08;8.31;9.87;228;4;180;3;FEW\_VIOLATIONS\_FOUND;2.22
qmstp\_CP25\_100\_100\_100;100;4;25;300;18370;15807.56;15808;0.14;13.95;2.29;15817.1;15818;0.14;13.9;12.11;139;4;170;2;FEW\_VIOLATIONS\_FOUND;0.37
qmstp\_CP30\_33\_100\_100;33;4;30;143;31077;29621.42;29622;0.05;4.68;1.31;29783.85;29784;0.04;4.16;14.2;325;3;539;4;FEW\_VIOLATIONS\_FOUND;11.16
qmstp\_CP30\_67\_100\_100;67;4;30;291;28777;25518.14;25519;0.11;11.32;1.87;25546.39;25547;0.11;11.23;11.24;271;3;160;3;SLOW\_IMPROVEMENT;0.87
qmstp\_CP30\_100\_100\_100;100;4;30;435;27314;22912.94;22913;0.16;16.11;24.38;22917.52;22918;0.16;16.1;145.29;93;2;1151;2;FEW\_VIOLATIONS\_FOUND;0.1
qmstp\_CP35\_33\_100\_100;33;4;35;196;43629;40846.14;40847;0.06;6.38;1.11;41011.82;41012;0.06;6;12.59;392;3;306;4;FEW\_VIOLATIONS\_FOUND;5.95
qmstp\_CP35\_67\_100\_100;67;4;35;398;39660;34968.45;34969;0.12;11.83;3.48;34986.17;34987;0.12;11.78;14.38;255;3;147;2;FEW\_VIOLATIONS\_FOUND;0.38
qmstp\_CP35\_100\_100\_100;100;4;35;595;38049;31466.73;31467;0.17;17.3;11.39;31469.65;31470;0.17;17.29;31.31;90;2;192;2;FEW\_VIOLATIONS\_FOUND;0.04
qmstp\_CP40\_33\_100\_100;33;4;40;257;57874;53626.37;53627;0.09;7.34;4.24;53763.25;53764;0.09;7.1;72.15;441;3;882;3;FEW\_VIOLATIONS\_FOUND;3.22
qmstp\_CP40\_67\_100\_100;67;4;40;522;53592;45901.66;45902;0.14;14.35;6.33;45914.47;45915;0.14;14.33;24.59;235;3;157;2;FEW\_VIOLATIONS\_FOUND;0.17
qmstp\_CP40\_100\_100\_100;100;4;40;780;51229;41229.31;41230;0.2;19.52;22.34;41231.11;41232;0.2;19.52;53.43;73;2;206;2;FEW\_VIOLATIONS\_FOUND;0.02
qmstp\_CP45\_33\_100\_100;33;4;45;326;72676;67799.91;67800;0.07;6.71;3.21;67966.4;67967;0.06;6.48;17.31;582;3;265;4;FEW\_VIOLATIONS\_FOUND;3.41
qmstp\_CP45\_67\_100\_100;67;4;45;663;68737;58407.29;58408;0.15;15.03;10.4;58412.38;58413;0.15;15.02;29.33;165;2;154;2;FEW\_VIOLATIONS\_FOUND;0.05
qmstp\_CP45\_100\_100\_100;100;4;45;990;66508;52425.3;52426;0.21;21.17;44.09;52425.78;52426;0.21;21.17;75.26;40;1;207;2;FEW\_VIOLATIONS\_FOUND;0
qmstp\_CP50\_33\_100\_100;33;4;50;404;91009;83901.38;83902;0.08;7.81;2.98;84027.09;84028;0.08;7.67;16.53;564;2;162;4;FEW\_VIOLATIONS\_FOUND;1.77
qmstp\_CP50\_67\_100\_100;67;4;50;820;86231;72050.92;72051;0.16;16.44;19.43;72055.97;72056;0.16;16.44;47.69;153;2;168;2;FEW\_VIOLATIONS\_FOUND;0.04
qmstp\_CP50\_100\_100\_100;100;4;50;1225;83838;65003.1;65004;0.22;22.47;80.94;65003.46;65004;0.22;22.47;127.26;33;2;222;2;FEW\_VIOLATIONS\_FOUND;0
\end{filecontents*}
\csvreader[head to column names, late after line = \\, /csv/separator=semicolon]{result-CP4.csv}{n=\size,m=\edges,d=\density}
        {\size & \density & \edges & \ub & \dnn & \dnngapp & \dnntime & \lb  & \lbgapp & \time & \iterations & \ncuts & \relgapclosed}
    \cmidrule(r){1-4} \cmidrule(lr){5-7} \cmidrule(l){8-13}        
    \multicolumn{4}{l|}{\revision{Average}} & & \revision{10.47} & & & \revision{9.84} & & & & \revision{12.65}
    \end{longtable}
\endgroup
}

    {\footnotesize
\begingroup
\setlength{\tabcolsep}{4pt} 
\renewcommand{\arraystretch}{1.1} 
    \begin{longtable}{lrrr|rrr|rrrrrr}
    \toprule
    \multicolumn{4}{c}{Instance} & \multicolumn{3}{c}{DNN} & \multicolumn{6}{c}{DNN + CUTS}\\
    \cmidrule(r){1-4} \cmidrule(lr){5-7} \cmidrule(l){8-13}
         \multicolumn{1}{l}{$n$}& \multicolumn{1}{l}{$d$ (\%)} & \multicolumn{1}{l}{$m$} & \multicolumn{1}{l}{UB} & \multicolumn{1}{l}{LB}  & \multicolumn{1}{l}{gap (\%)} & \multicolumn{1}{l}{time (s)} & \multicolumn{1}{l}{LB}  & \multicolumn{1}{l}{gap (\%)} & \multicolumn{1}{l}{time (s)} & \multicolumn{1}{l}{iterations} & \multicolumn{1}{l}{cuts} & \multicolumn{1}{l}{closed (\%)}  \\\midrule \endfirsthead
        \multicolumn{1}{l}{$n$}& \multicolumn{1}{l}{$d$ (\%)} & \multicolumn{1}{l}{$m$} & \multicolumn{1}{l}{UB} & \multicolumn{1}{l}{LB}  & \multicolumn{1}{l}{gap (\%)} & \multicolumn{1}{l}{time (s)} & \multicolumn{1}{l}{LB}  & \multicolumn{1}{l}{gap (\%)} & \multicolumn{1}{l}{time (s)} & \multicolumn{1}{l}{iterations} & \multicolumn{1}{l}{cuts} & \multicolumn{1}{l}{closed (\%)}\\\midrule\endhead
        \bottomrule\\
        \caption{Results for SV instances.}\label{tab:SV}\endlastfoot
        \begin{filecontents*}{result-SV.csv}
instance;n;d;m;ub;dnn;dnnrounded;dnngap;dnngapp;dnntime;lb;lbrounded;lbgap;lbgapp;time;ncuts;nclusters;iterations;outerloops;solutionstatus;relgapclosed
instance\_10\_33\_4;10;33;11;4217;4207.86;4208;0.00;0.22;0.01;4207.86;4208;0.00;0.21;0.01;0;0;105;1;FEW\_VIOLATIONS\_FOUND;0.00
instance\_10\_67\_1;10;67;25;3981;3960.66;3961;0.01;0.51;0.45;3974.47;3975;0.00;0.15;6.20;47;1;1155;3;FEW\_VIOLATIONS\_FOUND;67.88
instance\_10\_100\_4;10;100;45;3930;3906.64;3907;0.01;0.59;0.24;3914.25;3915;0.00;0.38;14.53;78;4;945;3;FEW\_VIOLATIONS\_FOUND;32.56
instance\_12\_33\_4;12;33;28;6141;6115.80;6116;0.00;0.41;0.08;6124.29;6125;0.00;0.26;2.15;42;2;619;3;FEW\_VIOLATIONS\_FOUND;33.70
instance\_12\_67\_4;12;67;51;6050;6018.40;6019;0.01;0.52;0.19;6030.75;6031;0.00;0.31;17.54;71;3;803;3;FEW\_VIOLATIONS\_FOUND;39.09
instance\_12\_100\_1;12;100;66;6051;6003.50;6004;0.01;0.79;0.32;6012.51;6013;0.01;0.63;22.41;144;6;1003;4;FEW\_VIOLATIONS\_FOUND;18.98
instance\_14\_33\_2;14;33;29;8736;8691.84;8692;0.01;0.51;0.12;8733.38;8734;0.00;0.02;4.85;81;2;1111;4;FEW\_VIOLATIONS\_FOUND;94.07
instance\_14\_67\_2;14;67;53;8606;8580.96;8581;0.00;0.29;0.25;8593.99;8594;0.00;0.14;11.48;133;3;956;4;FEW\_VIOLATIONS\_FOUND;52.03
instance\_14\_100\_2;14;100;91;8513;8458.85;8459;0.01;0.64;1.28;8482.15;8483;0.00;0.35;32.85;221;6;1206;4;FEW\_VIOLATIONS\_FOUND;43.03
instance\_16\_33\_5;16;33;36;11735;11685.13;11686;0.00;0.43;0.14;11706.79;11707;0.00;0.24;3.51;102;2;929;4;FEW\_VIOLATIONS\_FOUND;43.43
instance\_16\_67\_1;16;67;87;11610;11536.13;11537;0.01;0.64;1.07;11551.71;11552;0.00;0.50;12.39;217;4;949;4;FEW\_VIOLATIONS\_FOUND;21.09
instance\_16\_100\_3;16;100;120;11516;11468.40;11469;0.00;0.41;2.51;11480.52;11481;0.00;0.30;39.79;280;6;1675;4;FEW\_VIOLATIONS\_FOUND;25.46
instance\_18\_33\_1;18;33;51;15125;15059.14;15060;0.00;0.44;0.26;15093.78;15094;0.00;0.20;11.70;167;3;1253;5;FEW\_VIOLATIONS\_FOUND;52.59
instance\_18\_67\_5;18;67;108;15019;14947.73;14948;0.00;0.47;1.59;14958.75;14959;0.00;0.40;14.80;229;4;1101;4;FEW\_VIOLATIONS\_FOUND;15.46
instance\_18\_100\_4;18;100;153;14943;14883.24;14884;0.00;0.40;3.25;14889.90;14890;0.00;0.35;17.47;247;5;850;3;FEW\_VIOLATIONS\_FOUND;11.14
instance\_20\_33\_2;20;33;60;19057;18975.25;18976;0.00;0.43;0.28;19001.60;19002;0.00;0.29;9.09;177;3;1191;4;FEW\_VIOLATIONS\_FOUND;32.23
instance\_20\_67\_4;20;67;127;18830;18757.29;18758;0.00;0.39;2.40;18766.64;18767;0.00;0.33;9.66;245;4;1008;4;FEW\_VIOLATIONS\_FOUND;12.85
instance\_20\_100\_9;20;100;190;18812;18699.67;18700;0.01;0.60;2.31;18704.25;18705;0.01;0.57;14.02;261;5;398;3;FEW\_VIOLATIONS\_FOUND;4.08
instance\_25\_33\_3;25;33;104;30747;30660.41;30661;0.00;0.28;1.01;30685.44;30686;0.00;0.20;14.76;333;3;1127;5;FEW\_VIOLATIONS\_FOUND;28.91
instance\_25\_67\_2;25;67;195;30554;30416.45;30417;0.00;0.45;3.09;30435.81;30436;0.00;0.39;19.97;553;5;672;5;FEW\_VIOLATIONS\_FOUND;14.07
instance\_25\_100\_10;25;100;300;30405;30202.33;30203;0.01;0.67;4.64;30207.32;30208;0.01;0.65;13.46;263;4;271;3;FEW\_VIOLATIONS\_FOUND;2.46
instance\_30\_33\_2;30;33;152;45184;45048.50;45049;0.00;0.30;2.87;45077.39;45078;0.00;0.23;17.08;499;3;1056;5;FEW\_VIOLATIONS\_FOUND;21.32
instance\_30\_67\_2;30;67;273;44989;44756.82;44757;0.01;0.52;3.52;44765.07;44766;0.00;0.50;12.29;362;3;324;3;FEW\_VIOLATIONS\_FOUND;3.55
instance\_30\_100\_8;30;100;435;44847;44428.13;44429;0.01;0.93;4.66;44432.08;44433;0.01;0.92;13.72;227;3;138;3;FEW\_VIOLATIONS\_FOUND;0.94
\end{filecontents*}
\csvreader[head to column names, late after line = \\, /csv/separator=semicolon]{result-SV.csv}{n=\size,m=\edges,d=\density}
        {\size & \density & \edges & \ub & \dnn & \dnngapp & \dnntime & \lb  & \lbgapp & \time & \iterations & \ncuts & \relgapclosed}
    \cmidrule(r){1-4} \cmidrule(lr){5-7} \cmidrule(l){8-13}        
    \multicolumn{4}{l|}{\revision{Average}} & & \revision{0.49} & & & \revision{0.36} & & & & \revision{27.96}
    \end{longtable}
\endgroup
}

{\footnotesize
\begingroup
\setlength{\tabcolsep}{4pt} 
\renewcommand{\arraystretch}{1.1} 
    \begin{longtable}{lrr|rrr|rrrrrr}
    \toprule
    \multicolumn{3}{c}{Instance} & \multicolumn{3}{c}{DNN} & \multicolumn{6}{c}{DNN + CUTS}\\
    \cmidrule(r){1-3} \cmidrule(lr){4-6} \cmidrule(l){7-12}
         \multicolumn{1}{l}{$n$}& \multicolumn{1}{l}{$m$} & \multicolumn{1}{l}{UB} & \multicolumn{1}{l}{LB}  & \multicolumn{1}{l}{gap (\%)} & \multicolumn{1}{l}{time (s)} & \multicolumn{1}{l}{LB}  & \multicolumn{1}{l}{gap (\%)} & \multicolumn{1}{l}{time (s)} & \multicolumn{1}{l}{iterations} & \multicolumn{1}{l}{cuts} & \multicolumn{1}{l}{closed (\%)}\\\midrule \endfirsthead
        \multicolumn{1}{l}{$n$} & \multicolumn{1}{l}{$m$} & \multicolumn{1}{l}{UB} & \multicolumn{1}{l}{LB}  & \multicolumn{1}{l}{gap (\%)} & \multicolumn{1}{l}{time (s)} & \multicolumn{1}{l}{LB}  & \multicolumn{1}{l}{gap (\%)} & \multicolumn{1}{l}{time (s)} & \multicolumn{1}{l}{iterations} & \multicolumn{1}{l}{cuts} & \multicolumn{1}{l}{closed (\%)} \\\midrule\endhead
        \bottomrule\\
        \caption{Results for OPsym instances.}\label{tab:OPsym}\endlastfoot
        \begin{filecontents*}{result-OPsym.csv}
instance;n;m;ub;dnn;dnnrounded;dnngap;dnngapp;dnntime;lb;lbrounded;lbgap;lbgapp;time;ncuts;nclusters;iterations;outerloops;relgapclosed
sym-6;6;15;258.40;249.24;249.60;0.04;3.55;0.01;251.42;251.80;0.03;2.70;0.10;1.0;0.2;112.3;1.1;23.84
sym-7;7;21;326.80;307.63;307.90;0.06;5.87;0.06;315.81;316.10;0.03;3.36;11.19;12.8;1.4;454.4;1.8;42.68
sym-8;8;28;438.50;420.37;420.80;0.04;4.14;0.04;428.77;429.00;0.02;2.22;7.46;19.1;2.4;388.7;2.0;46.35
sym-9;9;36;534.90;505.99;506.40;0.05;5.40;0.06;525.22;525.50;0.02;1.81;18.09;33.6;3.4;523.0;2.6;66.50
sym-10;10;45;653.90;627.00;627.50;0.04;4.11;0.09;645.03;645.50;0.01;1.36;35.16;60.5;4.1;756.7;3.1;67.03
sym-11;11;55;785.90;747.23;747.70;0.05;4.92;0.13;764.22;764.70;0.03;2.76;45.54;71.8;4.3;784.2;3.2;43.92
sym-12;12;66;918.50;875.83;876.40;0.05;4.65;0.20;892.54;893.10;0.03;2.83;49.01;93.9;4.4;809.7;3.2;39.15
sym-13;13;78;1067.10;1018.91;1019.30;0.05;4.52;0.28;1036.33;1036.70;0.03;2.88;68.12;121.5;5.1;913.6;3.5;36.16
sym-14;14;91;1249.80;1200.89;1201.20;0.04;3.91;0.58;1220.08;1220.50;0.02;2.38;84.08;157.0;5.5;984.5;3.6;39.24
sym-15;15;105;1390.20;1327.62;1328.10;0.05;4.50;0.77;1348.64;1349.30;0.03;2.99;79.95;186.3;5.5;998.6;3.6;33.59
sym-16;16;120;1629.30;1542.28;1542.80;0.05;5.34;0.92;1567.60;1568.20;0.04;3.79;65.70;227.7;6.0;986.8;3.7;29.09
sym-17;17;136;1823.80;1723.14;1723.60;0.06;5.52;1.41;1746.23;1746.70;0.04;4.25;103.45;253.9;5.9;1046.7;3.8;22.93
sym-18;18;153;2981.00;1963.13;1963.50;0.34;34.15;1.80;1985.06;1985.60;0.33;33.41;104.87;279.8;5.3;950.1;3.9;2.15
sym-20;20;190;2572.70;2415.88;2416.30;0.06;6.10;2.61;2428.77;2429.40;0.06;5.59;99.94;299.3;5.7;846.0;3.8;8.22
sym-30;30;435;6015.90;5324.15;5324.70;0.11;11.50;10.38;5326.24;5326.70;0.11;11.46;102.24;228.9;3.4;426.9;2.6;0.30
sym-50;50;1225;17616.90;14104.59;14105.30;0.20;19.94;168.73;14104.71;14105.30;0.20;19.94;440.47;44.0;1.6;387.5;2.0;0.00
\end{filecontents*}
\csvreader[head to column names, late after line = \\, /csv/separator=semicolon]{result-OPsym.csv}{n=\size,m=\edges}
        {\size & \edges & \ub & \dnn & \dnngapp & \dnntime & \lb  & \lbgapp & \time & \iterations & \ncuts & \relgapclosed}
    \cmidrule(r){1-3} \cmidrule(lr){4-6} \cmidrule(l){7-12}        
    \multicolumn{3}{l|}{\revision{Average}} & & \revision{8.01} & & & \revision{6.48} & & & & \revision{31.32}
    \end{longtable}
\endgroup
}
{\footnotesize
\begingroup
\vspace{-0.5cm}
\setlength{\tabcolsep}{4pt} 
\renewcommand{\arraystretch}{1.1} 
        \begin{longtable}{lrr|rrr|rrrrrr}
    \toprule
    \multicolumn{3}{c}{Instance} & \multicolumn{3}{c}{DNN} & \multicolumn{6}{c}{DNN + CUTS}\\
    \cmidrule(r){1-3} \cmidrule(lr){4-6} \cmidrule(l){7-12}
         \multicolumn{1}{l}{$n$}& \multicolumn{1}{l}{$m$} & \multicolumn{1}{l}{UB} & \multicolumn{1}{l}{LB}  & \multicolumn{1}{l}{gap (\%)} & \multicolumn{1}{l}{time (s)} & \multicolumn{1}{l}{LB}  & \multicolumn{1}{l}{gap (\%)} & \multicolumn{1}{l}{time (s)} & \multicolumn{1}{l}{iterations} & \multicolumn{1}{l}{cuts} & \multicolumn{1}{l}{closed (\%)}\\\midrule \endfirsthead
        \multicolumn{1}{l}{$n$} & \multicolumn{1}{l}{$m$} & \multicolumn{1}{l}{UB} & \multicolumn{1}{l}{LB}  & \multicolumn{1}{l}{gap (\%)} & \multicolumn{1}{l}{time (s)} & \multicolumn{1}{l}{LB}  & \multicolumn{1}{l}{gap (\%)} & \multicolumn{1}{l}{time (s)} & \multicolumn{1}{l}{iterations} & \multicolumn{1}{l}{cuts} & \multicolumn{1}{l}{closed (\%)} \\\midrule\endhead
        \bottomrule\\
        \caption{Results for OPesym instances.}\label{tab:OPesym}\endlastfoot
        \begin{filecontents*}{result-OPesym.csv}
instance;n;m;UB;dnn;dnnrounded;dnngap;dnngapp;dnntime;lb;lbrounded;lbgap;lbgapp;time;ncuts;nclusters;iterations;outerloops;relgapclosed
esym-6;6;15;541.20;421.57;421.70;0.22;22.10;0.01;499.51;499.70;0.08;7.70;3.80;9.5;1.6;279.8;1.8;65.15
esym-7;7;21;783.70;568.79;569.10;0.27;27.42;0.03;743.62;744.00;0.05;5.11;27.81;24.3;2.4;526.2;2.6;81.35
esym-8;8;28;1020.10;706.73;707.20;0.31;30.72;0.04;956.62;956.80;0.06;6.22;55.29;45.6;3.5;548.3;3.1;79.74
esym-9;9;36;1356.00;1113.97;1114.30;0.18;17.85;0.05;1313.50;1314.00;0.03;3.13;74.30;57.9;3.5;695.3;3.4;82.44
esym-10;10;45;1427.10;1044.20;1044.60;0.27;26.83;0.07;1362.01;1362.30;0.05;4.56;118.12;84.2;3.9;643.2;3.3;83.00
esym-11;11;55;1545.10;1122.77;1123.10;0.27;27.33;0.12;1431.81;1432.00;0.07;7.33;223.74;100.0;5.2;873.6;3.8;73.18
esym-12;12;66;1901.60;1420.73;1421.10;0.25;25.29;0.13;1815.21;1815.50;0.05;4.54;432.63;160.9;6.0;993.1;4.5;82.03
esym-13;13;78;2175.30;1684.29;1684.80;0.23;22.57;0.18;2051.16;2051.60;0.06;5.71;430.93;174.7;6.7;943.8;3.9;74.72
esym-14;14;91;2527.90;1911.59;1912.10;0.24;24.38;0.46;2367.42;2367.70;0.06;6.35;720.47;256.4;7.3;1107.6;4.4;73.96
esym-15;15;105;2588.80;2173.45;2174.00;0.16;16.04;0.52;2426.66;2427.00;0.06;6.26;468.08;200.6;7.2;810.8;3.9;60.96
esym-16;16;120;2980.10;2360.16;2360.60;0.21;20.80;1.07;2765.42;2765.90;0.07;7.20;834.69;291.8;7.2;1214.3;4.2;65.37
esym-17;17;136;3372.20;2327.66;2328.20;0.31;30.98;1.00;3165.01;3165.50;0.06;6.14;1933.56;445.3;9.4;1678.0;5.4;80.16
esym-18;18;153;3569.00;2645.32;2645.60;0.26;25.88;1.06;3281.36;3281.90;0.08;8.06;1947.79;421.7;9.0;1403.7;4.7;68.86
%esym-20;20;190;;;;;;;;;;;;;;;;
esym-30;30;435;8056.70;6114.86;6115.30;0.24;24.10;17.30;7174.26;7174.80;0.11;10.95;10045.11;1056.8;13.2;1244.4;3.7;54.56
esym-50;50;1225;15788.80;13030.28;13030.90;0.17;17.47;406.93;13333.52;13334.10;0.16;15.55;10923.68;1225.0;15.8;800.5;2.0;10.99
\end{filecontents*}
\csvreader[head to column names, late after line = \\, /csv/separator=semicolon]{result-OPesym.csv}{n=\size,m=\edges}
        {\size & \edges & \UB & \dnn & \dnngapp & \dnntime & \lb  & \lbgapp & \time & \iterations & \ncuts & \relgapclosed}
    \cmidrule(r){1-3} \cmidrule(lr){4-6} \cmidrule(l){7-12}        
    \multicolumn{3}{l|}{\revision{Average}} & & \revision{23.98} & & & \revision{6.99} & & & & \revision{69.10}
    \end{longtable}
\endgroup
}
{\footnotesize
\begingroup
\setlength{\tabcolsep}{4pt}
\renewcommand{\arraystretch}{1.1}
        \begin{longtable}{lrr|rrr|rrrrrr}
    \toprule
    \multicolumn{3}{c}{Instance} & \multicolumn{3}{c}{DNN} & \multicolumn{6}{c}{DNN + CUTS}\\
    \cmidrule(r){1-3} \cmidrule(lr){4-6} \cmidrule(l){7-12}
         \multicolumn{1}{l}{$n$}& \multicolumn{1}{l}{$m$} & \multicolumn{1}{l}{UB} & \multicolumn{1}{l}{LB}  & \multicolumn{1}{l}{gap (\%)} & \multicolumn{1}{l}{time (s)} & \multicolumn{1}{l}{LB}  & \multicolumn{1}{l}{gap (\%)} & \multicolumn{1}{l}{time (s)} & \multicolumn{1}{l}{iterations} & \multicolumn{1}{l}{cuts} & \multicolumn{1}{l}{closed (\%)}\\\midrule \endfirsthead
        \multicolumn{1}{l}{$n$} & \multicolumn{1}{l}{$m$} & \multicolumn{1}{l}{UB} & \multicolumn{1}{l}{LB}  & \multicolumn{1}{l}{gap (\%)} & \multicolumn{1}{l}{time (s)} & \multicolumn{1}{l}{LB}  & \multicolumn{1}{l}{gap (\%)} & \multicolumn{1}{l}{time (s)} & \multicolumn{1}{l}{iterations} & \multicolumn{1}{l}{cuts} & \multicolumn{1}{l}{closed (\%)} \\\midrule\endhead
        \bottomrule\\
        \caption{Results for OPvsym instances.}\label{tab:OPvsym}\endlastfoot
        \begin{filecontents*}{result-OPvsym.csv}
instance;n;m;ub;dnn;dnnrounded;dnngap;dnngapp;dnntime;lb;lbrounded;lbgap;lbgapp;time;ncuts;nclusters;iterations;outerloops;relgapclosed
vsym-6;6;15;16273.90;14868.34;14868.40;0.09;8.64;0.01;14961.49;14961.50;0.08;8.06;0.14;1.0;0.2;56.8;1.1;6.63
vsym-7;7;21;19625.70;18483.49;18483.50;0.06;5.82;0.02;18483.49;18483.50;0.06;5.82;0.02;0.0;0.0;49.2;1.0;0.00
vsym-8;8;28;27039.40;23571.44;23571.50;0.13;12.83;0.02;25283.58;25283.80;0.06;6.49;22.61;9.1;1.2;757.0;1.6;49.37
vsym-9;9;36;22769.90;19600.00;19600.20;0.14;13.92;0.04;21514.48;21514.70;0.06;5.51;5.74;7.5;0.9;267.2;1.4;60.40
vsym-10;10;45;25743.80;22937.70;22938.00;0.11;10.90;0.06;24522.61;24523.00;0.05;4.74;19.27;17.4;2.0;494.3;1.9;56.48
vsym-11;11;55;29325.60;24918.41;24918.80;0.15;15.03;0.07;28930.89;28931.40;0.01;1.35;30.67;38.2;4.1;350.4;2.9;91.04
vsym-12;12;66;32577.80;29185.79;29186.10;0.10;10.41;0.13;32481.73;32482.10;0.00;0.29;115.64;43.4;4.0;674.7;3.0;97.17
vsym-13;13;78;40488.50;34246.69;34247.10;0.15;15.42;0.13;39782.08;39782.70;0.02;1.74;73.07;64.8;6.0;385.7;3.9;88.68
vsym-14;14;91;44240.40;38383.15;38383.50;0.13;13.24;0.52;44071.71;44072.20;0.00;0.38;156.78;64.4;5.9;797.9;3.3;97.12
vsym-15;15;105;50821.60;42399.17;42399.60;0.17;16.57;0.46;50113.95;50114.50;0.01;1.39;155.86;100.7;7.9;534.6;3.8;91.60
vsym-16;16;120;41940.20;35936.50;35937.10;0.14;14.31;1.26;41356.18;41356.70;0.01;1.39;290.43;85.7;5.9;1023.2;3.0;90.27
vsym-17;17;136;41819.00;32332.91;32333.50;0.23;22.68;0.79;41374.89;41375.30;0.01;1.06;592.39;143.8;9.3;1183.9;4.5;95.32
vsym-18;18;153;46130.20;36741.44;36741.70;0.20;20.35;1.27;45077.44;45078.00;0.02;2.28;645.68;143.0;8.6;1142.1;4.2;88.79
vsym-20;20;190;55326.20;41076.20;41076.60;0.26;25.76;2.08;55048.51;55049.10;0.01;0.50;1530.01;184.8;9.6;1362.4;3.8;98.05
vsym-30;30;435;78999.90;55898.20;55898.70;0.29;29.24;21.51;76507.06;76507.60;0.03;3.16;10781.08;426.6;15.2;1944.7;3.3;89.21
vsym-50;50;1225;165419.60;107324.67;107325.10;0.35;35.12;633.63;154950.21;154950.80;0.06;6.33;10835.89;1064.8;22.2;1787.8;2.0;81.98
\end{filecontents*}
\csvreader[head to column names, late after line = \\, /csv/separator=semicolon]{result-OPvsym.csv}{n=\size,m=\edges}
        {\size & \edges & \ub & \dnn & \dnngapp & \dnntime & \lb  & \lbgapp & \time & \iterations & \ncuts & \relgapclosed}
    \cmidrule(r){1-3} \cmidrule(lr){4-6} \cmidrule(l){7-12}        
    \multicolumn{3}{l|}{\revision{Average}} & & \revision{16.89} & & & \revision{3.16} & & & & \revision{73.88}
    \end{longtable}
\endgroup
}

\bibliography{QMST} 

\end{document}